\documentclass[11pt]{article}

\usepackage{thmtools}
\usepackage{thm-restate}
\usepackage{bm}
\usepackage{mathrsfs}           %
\usepackage{tikz}
\usetikzlibrary{positioning,chains,fit,shapes,calc}
\usepackage{algorithm}
\usepackage{algpseudocode}
\usepackage{mathtools}

\usepackage{subcaption}
\usepackage{enumerate}

\usepackage{amsmath,amssymb, amsthm}    %
\usepackage{verbatim}           %
\usepackage{xspace}             %
\usepackage{graphicx,float}     %
\usepackage{ifthen,calc}        %
\usepackage{textcomp}           %
\usepackage{fancybox}           %
\usepackage{hhline}             %
\usepackage{float}              %

\usepackage{booktabs}       %
\usepackage{amsfonts}       %
\usepackage{nicefrac}       %
\usepackage{microtype}      %

\usepackage{booktabs}
\usepackage{multirow}
\usepackage{bbm}
\usepackage{graphicx}
\usepackage{amssymb,amsmath}
\usepackage{epstopdf}
\usepackage{algorithm,algcompatible}
\usepackage{mathtools}
\usepackage{wrapfig}
\usepackage{soul}
\usepackage{tikz}

\usepackage[vflt]{floatflt}     %
\usepackage[small,compact]{titlesec}
\usepackage{setspace}

\usepackage{enumitem}
\setenumerate[1]{label={(\arabic*)}}
\usepackage[algo2e, ruled,vlined, boxed,linesnumbered]{algorithm2e}
\usepackage{color}
\definecolor{ForestGreen}{rgb}{0.2,0.2,0.6}
\usepackage{thumbpdf}
\usepackage[
colorlinks=true,
linkcolor=ForestGreen,
citecolor=ForestGreen,
bookmarks,bookmarksopen,bookmarksnumbered]
{hyperref}
\usepackage[margin = 1in]{geometry}
\newcommand{\acks}[1]{\section*{Acknowledgments}#1}
\newcommand{\showccc}[0]{0}
\newcommand{\ccc}[2][nothing]{%
	\ifthenelse{\showccc=0}{}{
		\ensuremath{^{\Lsh\Rsh}}\marginpar{\raggedright\tiny\textsf{%
				\ifthenelse{\equal{#1}{nothing}}{}{\textbf{#1}\\}#2}}}}

\usepackage{natbib}
\newtheorem{theorem}{Theorem}

\newtheorem{proposition}{Proposition}
\newtheorem{corollary}{Corollary}

\newtheorem{remark}{Remark}
\newtheorem{lemma}{Lemma}

\usepackage{url}

\mathtoolsset{showonlyrefs}

\newcommand{\mE}{{\mathbb E}}

\newcommand{\mR}{{\mathbb R}}

\newcommand{\cM}{{\mathcal M}}
\newcommand{\cN}{{\mathcal N}}
\newcommand{\cO}{{\mathcal O}}

\newcommand{\one}{{\mathbf 1}}

\def\TV{{\textup{TV}}}
\def\PI{{\textup{PI}}}
\def\LSI{{\textup{LSI}}}

\def\<{{\langle}}
\def\>{{\rangle}}

\def\d{{\rm{d}}}

\newcommand{\Pc}{Poincar\'e\xspace}
\newcommand{\Ren}{R\'enyi \xspace}

\newcommand{\Var}{{\text{Var}}}

\newcommand{\fan}[1]{{\color{black}{#1}}}

\newcommand{\ide}{\mathbf{I}}
\def\Rd{{\mathbb{R}^d}}

\def\Cconst{{(1+\alpha)\left(\frac{1}{\alpha}\right)^{\frac{\alpha}{1+\alpha}}\left(\frac{1}{\pi^2}\right)^{\frac{1}{1+\alpha}}2^{\frac{1-\alpha}{1+\alpha}}}}

\def\Ciconst{{(1+\alpha_j)\left(\frac{1}{\alpha_j}\right)^{\frac{\alpha_j}{1+\alpha_j}}\left(\frac{1}{\pi^2}\right)^{\frac{1}{1+\alpha_j}}2^{\frac{1-\alpha_j}{1+\alpha_j}}}}

\def\Cepshalf{{(1+\alpha)\left(\frac{1}{\alpha}\right)^{\frac{\alpha}{1+\alpha}}\left(\frac{1}{\pi^2}\right)^{\frac{1}{1+\alpha}}2^{\frac{1-2\alpha}{1+\alpha}}}}

\def\Ceta{{(1+\alpha)\left(\frac{1}{\alpha}\right)^{\frac{\alpha}{1+\alpha}}\left(\frac{1}{\pi^2}\right)^{\frac{1}{1+\alpha}}2^{\frac{-1-2\alpha}{1+\alpha}} (\log2)^{\frac{2}{1+\alpha}} }}

\def\Cjepshalf{{(1+{\alpha_j})\left(\frac{1}{{\alpha_j}}\right)^{\frac{{\alpha_j}}{1+{\alpha_j}}}\left(\frac{1}{\pi^2}\right)^{\frac{1}{1+{\alpha_j}}}2^{\frac{1-2{\alpha_j}}{1+{\alpha_j}}}}}

\def\Consteta{{49}}
\def\TwiceConsteta{{98}}
\def\RtConsteta{{7}}

\everypar{\looseness=-1}

\begin{document}

	\begin{titlepage}
		\def\thepage{}
		\thispagestyle{empty}
		
		\title{Improved dimension dependence of \\
  a proximal algorithm for sampling} 
		
		\date{}
		\author{
			Jiaojiao Fan \thanks{Georgia Institute of Technology, {\tt jiaojiaofan@gatech.edu}. First two authors contribute equally.}
			\and
			Bo Yuan\thanks{Georgia Institute of Technology, {\tt byuan48@gatech.edu}}
			\and
			Yongxin Chen\thanks{Georgia Institute of Technology, {\tt yongchen@gatech.edu}}
  }
		
		\maketitle
		
\abstract{
We propose a sampling algorithm that achieves superior complexity bounds in all the classical settings (strongly log-concave, log-concave, Logarithmic-Sobolev inequality (LSI), Poincar\'e inequality) as well as more general settings with semi-smooth or composite potentials. Our algorithm is based on the proximal sampler introduced in~\citet{lee2021structured}. The performance of this proximal sampler is determined by that of the restricted Gaussian oracle (RGO), a key step in the proximal sampler. The main contribution of this work is an inexact realization of RGO based on approximate rejection sampling. To bound the inexactness of RGO, we establish a new concentration inequality for semi-smooth functions over Gaussian distributions, extending the well-known concentration inequality for Lipschitz functions. Applying our RGO implementation to the proximal sampler, we achieve state-of-the-art complexity bounds in almost all settings. For instance, for strongly log-concave distributions, our method has complexity bound $\tilde\cO(\kappa d^{1/2})$ without warm start, better than the minimax bound for MALA. For distributions satisfying the LSI, our bound is $\tilde \cO(\hat \kappa d^{1/2})$ where $\hat \kappa$ is the ratio between smoothness and the LSI constant, better than all existing bounds. 
}
 		
	\end{titlepage}

\section{Introduction}

The task of sampling from a target distribution $\nu\propto \exp(-f) $ on $\mR^d$ plays an instrumental role in Bayesian inference~\citep{ghosal2017fundamentals}, scientific computing~\citep{pulido2019sequential}, and machine learning~\citep{murphy2012machine,liu2016stein,Fan2021VariationalWG}. 
Myriad works have been devoted to
the theoretical analysis of sampling, ranging from the smooth strongly log-concave setting~\citep{dalalyan2017theoretical,vempala2019rapid,durmus2019analysis} to non-log-concave~\citep{chewi2021analysis} or non-smooth settings~\citep{durmus2018efficient,salim2020primal,fan2022nesterov}.

In this work we make inroads towards better non-asymptotic complexity bound for sampling by focusing on 
the proximal sampler \citep{lee2021structured}. The proximal sampler is essentially a Gibbs sampler over an augmented distribution based on the target distribution. 
The difficulty of implementing the proximal sampler comes from restricted
Gaussian oracle (RGO) -- a task of sampling from $\exp(-f(\cdot) - \frac{1}{ 2\eta} \|x-y\|^2)$ for some given step size $\eta>0$ and $y \in \mR^d$. 
Given that RGO is implementable and exact, the proximal sampler can converge exponentially fast to the target distribution exponentially under mild assumptions~\citep{lee2021structured,chen2022improved}.
However, the total complexity of the algorithm heavily depends on the implementation of RGO. Except for some special settings \citep{gopi2022private}, the best-known dimension dependence of the proximal sampler is $\tilde \cO(d)$~\citep{chen2022improved,liang2022proximal}.
This is worse than the best-known bounds of other sampling methods such as underdamped Langevin Monte Carlo (ULMC)~\citep{shen2019randomized} or Metropolis-Adjusted Langevin Algorithm (MALA)~\citep{wu2021minimax}.

In this paper, we aim to improve the dimension dependence of the RGO and thus the proximal sampler. We first introduce an inexact RGO algorithm based on approximate rejection sampling. The underpinning of our analysis for this RGO algorithm is a novel Gaussian concentration inequality for semi-smooth functions, which extends the Gaussian concentration inequality for Lipschitz functions~\citep[Theorem 5.6]{boucheron2013concentration}. Our proof is based on the argument of Maurey and Pisier.
Our RGO algorithm is an inexact algorithm, and the crux is to bound the step size such that the output of our RGO is close enough to the exact RGO.
The RGO step size is then processed to only depend on the order of the aforementioned concentration inequality.

\paragraph{Contribution}
First, we propose a novel realization of RGO. 
Our RGO implementation is inexact but can achieve a better step size pertaining to the dimension. Next,
we prove a Gaussian concentration inequality for semi-smooth functions, which could be of independent interest. 
It can recover the order of the well-known Gaussian concentration inequality for Lipschitz functions.
This concentration inequality is the underpinning of our RGO analysis.
Third, we control the accumulated error from our inexact RGO algorithm in terms of both total variation and Wasserstein metric. This, combined with the existing proximal sampler convergence results, gives state-of-the-art sampling convergence results under various conditions
(see Table \ref{tab:t1}). Finally, we extend all the results for semi-smooth potential \eqref{eq:semi-smooth} to composite potential \eqref{eq:composite} (see Table \ref{tab:t2}).

\begin{table}[H]
\caption{Complexity bounds for sampling from semi-smooth potentials satisfying \eqref{eq:semi-smooth}. Here $ L_1, \delta,C_\LSI, C_\PI, \cM_2, \cM_4$ denote the smoothness constant, accuracy, LSI constant, \Pc inequality constant,  second moment, and fourth moment of the target distribution. 
}\label{tab:t1}
\begin{centering}
{\renewcommand{\arraystretch}{1.4}%
\begin{tabular}{|c|c|c|c|}
\hline
            Assumption           & Source                & Complexity                      & Metric                \\ \hline
\multirow{4}{*}{\begin{tabular}[c]{@{}c@{}} $\beta$-strongly \\ log-concave\end{tabular}} & \citet{chen2022improved}                 & $\tilde \cO(L_1 d/\beta)$                 &      \Ren                  \\ \cline{2-4} 
                                      &  \citet{wu2021minimax} & $\tilde \Omega (L_1 \sqrt{d}/\beta)$ (Warm start) & TV \\ \cline{2-4} 
                                      &  \citet{shen2019randomized} & $\tilde \cO  \left(
    (\frac{L_1}{\beta})^{7/6} ( \frac{2}{\delta} \sqrt{\frac{d}{\beta} } )^{1/3}
    +
    \frac{L_1}{\beta} ( \frac{2}{\delta } \sqrt{\frac{d}{\beta} } )^{2/3}
    \right) $  & $W_2$ \\ \cline{2-4}                                       
                                      & Proposition \ref{prop:strongly_convex}, \ref{prop:chi_strongly_conv_appro} & $\tilde \cO(L_1 \sqrt{d}/\beta)$  & TV/$W_2$/$\chi^2$ \\ \hline                                      
\multirow{2}{*}{log-concave}        & \citet{liang2022proximal}                      &   $\tilde \cO(\sqrt{\cM_4} L_\alpha^{\frac{2}{\alpha+1}} d /\delta  )$                    &  TV                    \\ \cline{2-4} 
                                      & Proposition \ref{prop:convex} 
                                      & $\tilde \cO(\cM_2 L_\alpha^{ \frac{2}{\alpha+1} } d^{\frac{\alpha}{\alpha+1}} / \delta )$ & TV \\ \hline
\multirow{2}{*}{log-Sobolev}        & \citet{chen2022improved}                      &   $\tilde \cO( L_1 d /C_\LSI  )$                    &  \Ren                    \\ \cline{2-4} 
                                      & Proposition \ref{prop:lsi}, \ref{prop:chi_appro_all} & $\tilde \cO( L_1 \sqrt{d} /C_\LSI )$ & TV/$\chi^2$ \\ \hline    
\multirow{2}{*}{\Pc}        & \citet{Liang2022APA}                      &   $\tilde \cO( {L_\alpha^{\frac{2}{\alpha+1}}d^2 }/{C_\PI}  )$                    &  \Ren                    \\ \cline{2-4} 
                                      & Proposition \ref{prop:pi}, \ref{prop:chi_appro_all} & $\tilde \cO( {L_\alpha^{\frac{2}{\alpha+1}}d^{\frac{2\alpha+1}{\alpha+1}} }/{C_\PI} )$ & TV/$\chi^2$ \\ \hline                                       
\end{tabular}
}
\par\end{centering}
\end{table}

\paragraph{Related works}
Our RGO algorithm is inspired by a rejection sampling-based RGO~\citep{gopi2022private} for distributions with Lipschitz potentials. Compared to their algorithm, we add a linear function to the potential to have a delicate stationary point,
and we simplify the rejection rule by eliminating part of the randomness.
The convergence of the proximal sampler is established for strongly log-concave distributions in \citet{lee2021structured}. \citet{chen2022improved} then extended the class of target distributions to a much wilder range, including \Pc inequality. Some other works~\citep{liang2021proximal,liang2022proximal,Liang2022APA,gopi2022private} consider the convergence under weaker smoothness conditions, e.g. semi-smooth potential or composite potential. We also mention a concurrent work \citep{Altschuler2023} that achieves similar complexity bounds as ours for smooth potentials with a very different RGO implementation based on MALA and ULMC. 
Other than the analysis for the proximal sampler, there exist numerous works for other sampling methods. To name a few, ~\citet{wu2021minimax,shen2019randomized} study strongly-log-concave and smooth potential, ~\citet{chewi2021analysis, erdogdu2021convergence,erdogdu2022convergence} study the non-log-concave potential, and \citet{nguyen2021unadjusted,durmus2018efficient,durmus2019analysis,salim2020primal,bernton2018langevin} study the composite potential.
More detailed discussions appear in \S \ref{sec:convergence}. 

\paragraph{Comparison to the cocurrent work \citep{Altschuler2023}}
\textbf{Algorithm}: We both use the proximal sampler but with different implementations of RGO. Our RGO is based on approximate rejection sampling and has complexity $\mathcal{O}(1)$. In contrast, their RGO is based on MALA with a ULMC warm start and has complexity $\tilde{\mathcal{O}}(\sqrt{d})$. Our RGO is easier in implementation and parameter tuning.
\textbf{Results}: For log-smooth distribution, we share the same complexity results in $\chi^2$.
Our results cover the semi-smooth and composite potentials, but they only cover log-smooth distributions. Thus, our complexity results are more general and include theirs as a special case. \textbf{Contributions}: The main contribution of our work includes a new RGO implementation and a new concentration inequality. The main contribution of \citet{Altschuler2023} is a warm start result using ULMC.  One disadvantage of our result is that it is not clear how our method can be used outside the proximal sampler framework, but their warm start method is applicable in various algorithms. 

\begin{table}[H]
\caption{Complexity bounds for sampling from composite potentials satisfying \eqref{eq:composite}.}\label{tab:t2}
\begin{centering}
{\renewcommand{\arraystretch}{1.4}%
\begin{tabular}{|c|c|c|c|}
\hline
            Assumption           & Source                & Complexity                      & Metric                \\ \hline
\multirow{3}{*}{log-Sobolev}        & \citet{nguyen2021unadjusted}                      &   $\tilde \cO( n \max\{L_{\alpha_j}^2\} d/\delta)^{\max\{ 1/\alpha_j\} } /C_\LSI^{1+ \max \{1/\alpha_j \} }  )$                    &  KL                   \\ \cline{2-4} 
  & \citet{Liang2022APA}                      &   $\tilde \cO(\sum_{j=1}^n L_{\alpha_j}^{2/(\alpha_j +1)} d /C_\LSI  )$                    &  KL                   \\ \cline{2-4} 
                                      & Proposition \ref{prop:composite} & $\tilde \cO\left( {   \left( \sum_{j=1}^n L_{\alpha_j}^{{1}/{(\alpha_j+1)}} d^{ {\alpha_j}/{ (2(\alpha_j+1)) } } \right)^2 }/{ C_\LSI }  \right) $ & TV \\ \hline    
\multirow{2}{*}{\Pc}        & \citet{Liang2022APA}                      &   $\tilde \cO( \sum_{j=1}^n L_{\alpha_j }^{2/(\alpha_j +1)} d^2 /{C_\PI}  )$                    &  \Ren                    \\ \cline{2-4} 
                                      & Proposition \ref{prop:composite} & $\tilde \cO\left( {   \left( \sum_{j=1}^n L_{\alpha_j}^{{1}/{(\alpha_j+1)}} d^{ {\alpha_j}/{ (2(\alpha_j+1)) } } \right)^2 }/{ C_\PI }  \right) $ & TV \\ \hline                                       
\end{tabular}
}
\par\end{centering}
\end{table}

\paragraph{Organization}
The paper is organized as follows. In \S \ref{sec:proximal}, we review the proximal sampler. We then introduce our RGO implementation and present the improved dimension dependence of RGO in \S \ref{sec:exact_rgo}. We establish the sampling convergence under different conditions in \S \ref{sec:convergence}. 
In \S \ref{sec:composite}, we extend all the analysis (\S \ref{sec:exact_rgo}-\ref{sec:convergence}) for semi-smooth potentials to the composite potentials.
We conclude in \S \ref{sec:conclude} with a discussion of future research directions.

\section{Background: the proximal sampler}\label{sec:proximal}

We consider sampling from the distribution $\nu \propto \exp(-f(x))$
where the potential $f$ is bounded from below and is $L_\alpha$-$\alpha$-semi-smooth, i.e., $f$ satisfies
\begin{align}\label{eq:semi-smooth}
\|f'(u) - f'(v) \| \le L_\alpha \|u -v\|^\alpha, ~~\forall u,v \in \mR^d
\end{align}
for $L_\alpha>0$ and $\alpha \in [0,1]$. Here $f'$ represents a subgradient of $f$. 
When $\alpha>0$, this subgradient can be replaced by the gradient.
The condition \eqref{eq:semi-smooth} implies $f$ is $L_1$-smooth when $\alpha=1$ and a Lipschitz function satisfies \eqref{eq:semi-smooth} with $\alpha=0$.

\begin{algorithm2e}
\caption{The proximal Sampler \citep{lee2021structured} \label{algo:AlternatingSampler}}

\textbf{Input:} Target distribution $\exp(-f(x))$, step size $\eta>0$, initial
point $x_{0}$

\For{$t = 1,\ldots, T$}{

Sample $y_t \sim \pi^{Y|X} (y|x_{t-1}) \propto \exp(- \frac{1}{2\eta} \|x_{t-1} - y\|^2 ) $

Sample $x_t \sim \pi^{X|Y} (x|y_t) \propto \exp(-f (x)-\frac{1}{2\eta}\|x-y_{t}\|^{2})$ \label{line:sample_oracle}

}

\textbf{Return} $x_{T}$

\end{algorithm2e}

The algorithm we adopt is the proximal sampler (or the alternating sampler) proposed by \citet{lee2021structured}, shown in Algorithm \ref{algo:AlternatingSampler}.
It is essentially a Gibbs sampling method for the joint distribution $\pi^{XY} (x,y) \propto \exp\left(  -f(x) - \frac{1}{2\eta} \|x-y\|^2 \right)$.
The target distribution $\nu$ is the $X$-marginal distribution of $\pi^{XY}$.
The proximal sampler alternates between two 
 sampling steps. The first one is to sample from the conditional distribution of $Y$ given $x_{t-1}$; it is a Gaussian distribution $\pi^{Y|X}(y|x_{t-1}) = \cN(x_{t-1}, \eta \ide)$ and thus trivial to implement. The paramount part of this method is the 
second sub-step,
which is the restricted Gaussian oracle for $f$ to sample from the conditional distribution
\begin{align}
    \pi^{X|Y}(x|y_t) \propto \exp\left(-f(x) - \frac{1}{2\eta} \|x-y_t\|^2 \right),
\end{align}
which is not always easy to implement. To simplify the notation, we often use $\pi^{X|Y}$ to represent $\pi^{X|Y}(x|y)$.
If line \ref{line:sample_oracle} can be implemented exactly, then the proximal sampler is unbiased because the iterates $\{(x_t,y_t)\}_{t \in \mathbb{N}}$ form a reversible Markov chain with the stationary distribution $\pi^{XY}$. 

\subsection{Convergence of proximal sampler given exact RGO}\label{sec:proximal_convergence}
Next we briefly discuss the convergence property of Algorithm \ref{algo:AlternatingSampler}.
Recall that a probability distribution $\nu$ satisfies log-Sobolev inequality (LSI) with constant $C_\LSI>0$ ($C_\LSI$-LSI) if for all smooth functions $u: \mR^d \rightarrow \mR$, the following holds:
\begin{align}\tag{LSI} \label{eq:lsi}
    \mE_\nu[u^2 \log{u^2} ] -\mE_\nu[u^2] \log \mE_\nu[u^2] \le \frac{2}{C_\LSI} \mE_\nu [ \|\nabla u\|^2].    
\end{align}
Denote Wasserstein-2 distance as $W_2(\nu,\mu ) := \inf_{\gamma \in \Pi (\nu,\mu) } \int \|x-y\|^2 \d \gamma (x,y)  $, where $\Pi(\nu,\mu )$ is the set of joint distributions of marginal distributions $\nu$ and $\mu$. 
For a probability measure $\mu \ll \nu, $ we define the KL divergence $H_\nu(\mu ) := \int \mu \log \frac{\mu}{\nu}$, and the chi-squared divergence $\chi_\nu^2 (\mu) := \int \frac{\mu^2}{\nu} -1 $.

A distribution $\nu$ satisfies \eqref{eq:lsi} implies that it also satisfies Talagrand inequality~\citep{otto2000generalization}, i.e., $W_2(\nu,\mu) \le \sqrt{\frac{2}{C_\LSI} H_\nu ( \mu)}$ for any probability distribution $\mu \ll \nu$ with finite second moment.
A probability distribution $\nu$ satisfies the \Pc inequality (PI) with constant $C_\PI >0 $ if for any smooth bounded function $u: \mR^d \rightarrow \mR$, it holds that
\begin{align}\tag{PI}
\label{eq:pi}
\Var_\nu (u) \le \frac{1}{C_\PI} \mE[\|\nabla u\|^2 ].
\end{align}
The \eqref{eq:lsi} implies the \eqref{eq:pi} with the same constant. 
In the following theorem, we assume $x_0$ is sampled from some initialization distribution $\mu_0$. 

\begin{theorem}[Convergence of proximal sampler \citep{chen2022improved}]\label{thm:proximal_convergence}
Assuming 
 the RGO is exact, we denote the corresponding iterates $y_t \sim \psi_t $ and $x_t \sim \mu_t $. 

 1) If $\nu$ is log-concave (i.e. $f$ is convex), $H_\nu (\mu_t) \le {W_2^2 (\mu_0, \nu )}/{(t \eta)}; $

 2) If $\nu$ satisfies $C_\LSI$-\ref{eq:lsi}, $H_\nu (\mu_t) \le {H_\nu (\mu_0 )}/{(1+ C_\LSI \eta)^{2t}};$

 3) If $\nu$ satisfies $C_\PI$-\ref{eq:pi}, $   \chi^2_\nu(\mu_t) \le {\chi^2_\nu(\mu_0)}/{(1+ C_\PI \eta)^{2t} }.$
\end{theorem}

Note that if $\nu$ is $\beta$-strongly-log-concave, then $\nu$ satisfies $\beta$-LSI. So the convergence of strongly-log-concave $\nu$ is also implicitly contained in Theorem \ref{thm:proximal_convergence}.

\subsection{Existing RGO implementations}
Except for a few special cases with closed-form realizations \citep{lee2021structured,mou2022efficient}, most of the existing implementations of RGO are based on rejection sampling. 
Denote the function $f^\eta_y: = f(x) + \frac{1}{2 \eta} \|x-y\|^2$.
If $\pi^{X|Y}$ is strongly-log-concave, i.e., $f^\eta_y$ is strongly-convex and
smooth, 
one can naturally use a Gaussian with variance being the convexity of $f_y^\eta$ and mean being the minimizer of $f_y^\eta$ as the proposal distribution. In this case, 
with the step size $\eta = \Theta (1/(L_1 d))$, the expected number of iterations for rejection sampling is $\cO(1)$~\citep{liang2021proximal,chewi2022query}.
When the potential $f$ is not smooth but $L_0$-Lipschitz, \citet{liang2021proximal} shows that a similar proposal gives $\cO(1)$ complexity for rejection sampling when the step size is $\eta = \Theta (1/(L_0^2 d))$. In fact, it can be shown that, at least for quadratic potentials, the dependence $\eta = \Theta (1/d)$ is inevitable to ensure $\cO(1)$ complexity for rejection sampling-based RGO.
On the other hand, \citet{gopi2022private} proposes an approximate rejection sampling scheme for an inexact RGO for Lipschitz $f$ which uses a larger step size $\eta = \tilde \cO (1/L_0^2)$ to ensure $\cO(1)$ complexity of RGO, rendering better dimension dependence of the proximal sampler. On a different route, one can apply any MCMC algorithm to implement RGO with dimension-free step size $\eta$, but the complexity of each RGO step is dimension dependent, thus, the overall complexity of this strategy is not necessarily better \citep{lee2021structured}.

\section{Improved dimension dependence of RGO for semi-smooth potential}\label{sec:exact_rgo}

The step size $\eta$ of the proximal sampler is pivotal to the total complexity. As can be seen from Theorem \ref{thm:proximal_convergence} a larger step size points to a faster convergence rate. However, larger $\eta$ also means higher complexity for each RGO step. In \citet{Liang2022APA}, it is shown that, with $\eta = \Theta (1/(L_\alpha^{\frac{2}{\alpha+1}} d))$, it is possible to use rejection sampling to realize RGO with $\cO(1)$ complexity. In this section, we improve the step size to $\eta = \tilde\cO (1/(L_\alpha^{\frac{2}{\alpha+1}} d^{\frac{\alpha}{\alpha+1}}))$ while maintaining the same $\cO(1)$ complexity by designing a new RGO implementation based on approximate rejection sampling.

\begin{algorithm2e}[tbh]
\caption{Approximate rejection sampling implementation of RGO} \label{alg:rgo}

\textbf{Input}: $L_\alpha$-$\alpha$-semi-smooth function $f(x)$, step size $\eta>0$, current
point $y$

Compute $x_y$ such that $ f'(x_y) + \frac{1}{\eta} (x_y - y) =0 $. \label{step:opt}
Denote $g(x) = f(x) - \< f'(x_y) , x \>$. \label{step:proximal_opt}

\Repeat{$u\leq\frac{1}{2}\rho$}{

Sample $x,z$ from the distribution $ 
\phi (\cdot) 
\propto\exp( -\frac{1}{2\eta}\|\cdot - x_y \|^{2}_2)$

$\rho = \exp(g(z) - g(x))$

Sample $u$ uniformly from $[0,1]$.

}

\textbf{Return} $x$

\end{algorithm2e}

Our proposed RGO implementation is in Algorithm \ref{alg:rgo}. This
is a variant of \citet[Algorithm 2]{gopi2022private}. Compared to \citet{gopi2022private}, we eliminate the randomness in the inner loop of rejection sampling. 
More importantly, we modify the proposal distribution to ensure that the mean of the proposal Gaussian distribution $\phi$ is the same as the stationary point of the function $f_y^\eta$. This modification is justified by the following lemma. This modification is crucial to control the accuracy and complexity of our RGO algorithm  via concentration inequality as will be seen in Theorem \ref{thm:concentration}. 

\begin{lemma}\label{lem:equi_rgo}
Sampling from $ \pi^{X|Y} (x|y) \propto \exp(-f (x)-\frac{1}{2\eta}\|x-y\|^{2})$ is equivalent to sampling from distribution $\propto \exp (-  g(x) - \frac{1}{2 \eta} \|x-x_y \|^2 )$, where $g(x) = f(x) - \< f'(x_y) , x \> $ and $x_y$ satisfies $ f'(x_y) + \frac{1}{\eta } (x_y - y) = 0$.
\end{lemma}

Thanks to Lemma \ref{lem:equi_rgo}, 
sampling from distribution $\propto \exp (-  g(x) - \frac{1}{2 \eta} \|x-x_y \|^2 )$ is equivalent to the original RGO. Note that $g(x)$ shares the same semi-smooth constant as $f(x)$.
Algorithm \ref{alg:rgo} requires the calculation of the stationary point $x_y$ of $f(x) + \frac{1}{2\eta} \|x-y\|^2$. 
This could be challenging when $f$ is non-convex. In that case, we can get an approximate stationary point instead. We can control the error introduced by approximate $x_y$ and achieve the same complexity bound; the detailed discussions are postponed to \S \ref{sec:approximate_opt}. 
In the main paper, we assume the stationary point is solved exactly to make the analysis more comprehensible. To facilitate the analysis of Algorithm \ref{alg:rgo}, we introduce another threshold $\bar \rho:=\min(\rho,2)  $. As we will see, the acceptance probability
in our RGO algorithm
is directly related to $\bar \rho$.

\begin{lemma}\label{lem:dist_relation}
    Denote $\hat \pi^{X|Y}$ as the distribution of the output of Algorithm \ref{alg:rgo}. Let $\phi$ be defined as in Algorithm \ref{alg:rgo}. Define the random variables 
    $\bar \rho: =\min\left(\rho ,2\right)$,
    $V:= \mE [\rho | x],$ and $\overline V: = \mE [ \bar \rho | x]$. Then
    \begin{align}
 \frac{\d \pi^{X|Y} }{ \d x} = \frac{\d \phi}{\d x} \cdot \frac{\exp(-g(x))}{ \mE_{x \sim \phi} \exp(-g(x)) }  
& = \frac{\d \phi}{\d x} \cdot \frac{ \mE [\rho | x]}{ \mE [\rho] } 
 = \frac{\d \phi}{\d x} \cdot \frac{  V }{ \mE [ V ] } , \\
 \frac{\d \hat \pi^{X|Y} }{ \d x}
& = \frac{\d \phi}{\d x} \cdot \frac{ \mE [ \bar \rho | x]}{ \mE [ \bar \rho] } 
= \frac{\d \phi}{\d x} \cdot \frac{  \overline V }{ \mE [ \overline V ] }  .
    \end{align}
    Moreover, the acceptance probability of rejection sampling is $\frac{1}{2}\mE[\bar \rho] = \frac{1}{2}\mE[\overline V]  .$ 
\end{lemma}

Lemma \ref{lem:dist_relation} shows that the gap between ground truth $\pi^{X|Y}$ and our return $\hat \pi^{X|Y}$ is caused by the discrepancy between $\rho$ and $\bar \rho$. At a high level, we need to control the probability that $\rho$ and $\bar \rho$ are different by choosing a small enough step size. If $\eta$ is sufficiently small, then the proposal distribution $\phi$ has a small variance and so that $\rho$ is concentrated around 1 with high probability, which means $\rho$ and $\bar \rho$ are equal with high probability.
To give such a probabilistic bound, we 
establish
a
novel concentration inequality bound for a semi-smooth function over the Gaussian random variable.

\begin{theorem}[Gaussian concentration inequality for semi-smooth functions] \label{thm:concentration} 
Let $X \sim \cN(m,\eta\ide)$ be a Gaussian random variable in $\Rd$, and let $\ell$ be an $L_\alpha$-$\alpha$-semi-smooth function. Assume $ \ell'(m) = 0$. Then for any $r>0, ~0 \le \alpha \le 1$, one has 
\begin{align} 
\Pr( \ell(X)-\mE(\ell(X))\geq r) 
& \leq \left(1-\frac{\epsilon}{d} \right)^{-d/2} \exp\left(-\frac{C\epsilon^{\frac{\alpha}{1+\alpha}}r^{\frac{2}{1+\alpha}}}{L_\alpha^{\frac{2}{1+\alpha}}d^{\frac{\alpha}{1+\alpha}}\eta}\right), ~~ \forall \epsilon \in (0,d) \label{eq:alpha01_conc} \\
\text{where} ~~~~~~~~~~~~~~~
C & = \Cconst. \label{eq:conce_C}
\end{align}
\end{theorem}
\textbf{Sketch of proof for Theorem \ref{thm:concentration}}:
By Maurey and Pisier argument \citep{pisier2006probabilistic} and Young's inequality $\|G\|^{2\alpha} \leq \alpha\|G\|^2/\omega + (1-\alpha)\omega^{\frac{\alpha}{1-\alpha}} $ for $\forall~  \omega>0$, we have 
\begin{align}
 \Pr(\ell(X)-\mE(\ell(X)) \geq r) 
 \leq &  \inf_{\lambda >0} \frac{\mE_{G}\exp(\frac{\pi^2}{8} L_{\alpha}^2\eta\lambda^2\|G\|_2^{2\alpha})}{\exp(\lambda r)} \\
\le &\frac{\mE_{G}\exp\left(\frac{\pi^2}{8} L_{\alpha}^2\eta\lambda^2(\alpha\|G\|^2/ \omega + (1-\alpha) \omega^{\frac{\alpha}{1-\alpha}})\right)}{\exp(\lambda r)}.
\end{align}
Invoking the closed-form second moment of Gaussian distribution and choosing proper $\lambda$ and $\omega$, we can establish the result. The full proof is in \S \ref{sec:conc_proof}.

In Theorem \ref{thm:concentration}, $\epsilon$ is a tunable parameter, and we leave its choice to the user. It causes a trade-off between the coefficients $\left(1-\frac{\epsilon}{d} \right)^{-d/2}$, which is in front of $\exp()$, and $\epsilon^{\frac{\alpha}{1+\alpha}}$, which is inside $\exp()$. 
We remark that the assumption $\ell'(m)=0$ can be relaxed at the cost of an additional penalty coefficient (see Proposition \ref{prop:errored_concentration}). We also note that when $r$ is in a small range, we can always get sub-Gaussian tail no matter what the $\alpha$ value is (see Proposition \ref{prop:low_range}).

\begin{remark}
One term in the coefficient \eqref{eq:conce_C} in Theorem \ref{thm:concentration} is not well-defined when $\alpha=0$, but $C \rightarrow C_0 = 2/\pi^2$ as $\alpha \rightarrow 0$ and $C$ is monotone w.r.t. $\alpha$. Thus, for any $0\le \alpha \le 1$, we have
\begin{align} 
\Pr( \ell(X)-\mE(\ell(X))\geq r) 
\leq \left(1-\frac{\epsilon}{d} \right)^{-d/2} \exp\left(-\frac{2 \epsilon^{\frac{\alpha}{1+\alpha}}r^{\frac{2}{1+\alpha}}}{ \pi^2 L_\alpha^{\frac{2}{1+\alpha}}d^{\frac{\alpha}{1+\alpha}}\eta}\right), ~~ \forall \epsilon \in (0,d).
\end{align}
\end{remark}

Our concentration bound \eqref{eq:alpha01_conc} can recover the known bounds in two extreme cases. When $\alpha=0$, $\ell$ is Lipschitz, the RHS of \eqref{eq:alpha01_conc} is $\propto \exp(- r^2/(L_0^2 \eta))$. This recovers the sub-Gaussian tail for Lipschitz functions~\citep{ledoux1999concentration,boucheron2013concentration}.
When $\alpha=1$, the RHS of \eqref{eq:alpha01_conc} is $\propto \exp(- r /(L_1 \sqrt{d} \eta))$, which recovers the sub-exponential tail in the Laurent-Massart bound~\citep{laurent2000adaptive} for $\chi^2$ distribution and  Hanson-Wright inequality~\citep{hanson1971bound,wright1973bound,rudelson2013hanson} for large enough $r$.
A more detailed discussion 
is in \S \ref{sec:conc_relation}.

With this concentration inequality, we are able to find a small enough $\eta$ such that $\rho$ and $\bar \rho$ are the same with high probability, and thus the difference between $\hat \pi^{X|Y}$ and $\pi^{X|Y}$ is small.
 \begin{theorem}[RGO complexity in total variation]\label{thm:rgo}
Assume $f$ satisfies \eqref{eq:semi-smooth}. For $\forall \zeta>0$,
if
$$\eta \le 
\left(\Consteta L_\alpha^{\frac{2}{\alpha+1}}d^{\frac{\alpha}{\alpha+1}}(1+ \log (1+ 12/\zeta )) \right)^{-1}
,$$
then Algorithm \ref{alg:rgo} returns a random point $x$ that has $\zeta$ total variation distance to the distribution proportional to $
\pi^{X|Y}(\cdot | y)
$. Furthermore, if $~0 < \zeta < 1$, then the algorithm access $\cO(1)$ many $f(x)$ in expectation.
 \end{theorem}
 
 \begin{theorem}[RGO complexity in Wasserstein distance]\label{thm:rgo_w2}
Assume $f$ satisfies \eqref{eq:semi-smooth}. For $\forall \zeta>0$, if 
\begin{align}
\eta \le \min \left(  
\left(\Consteta L_\alpha^{\frac{2}{\alpha+1}}d^{\frac{\alpha}{\alpha+1}}(2+ \log( 1+ {192  (d^2 + 2d)}/{ \zeta^4} )) \right)^{-1}
, 1\right)  ,  
\end{align}
then Algorithm \ref{alg:rgo} returns a random point $x$ that has $\zeta$ Wasserstein-2 distance to the distribution $\pi^{X|Y}(\cdot | y)$.  Furthermore, if 
 ~$0< \zeta < 2\sqrt{2d} $, then the algorithm access $\cO(1)$ many $f(x)$ in expectation.
 \end{theorem}
{\bf Sketch of proof for Theorem \ref{thm:rgo} and Theorem \ref{thm:rgo_w2}:}
By definition and some elementary inequalities, the error $ \| \pi^{X|Y} - \hat \pi^{X|Y} \|_\TV $ can be bounded by $2\mE[\rho \one_{\rho \ge 2} ] \le 2 \sum_{i=1}^\infty \exp((i+1) \log 2 ) \Pr( (g(z) -g(x))/\log 2 \ge i) $, which is then bounded by applying our concentration inequality and choosing $\eta = \tilde \cO( 1/(  L_\alpha^{\frac{2}{\alpha+1}}d^{\frac{\alpha}{\alpha+1}} )  ) $. \fan{Specifically, we apply our concentration inequality to $\ell( x ) = g(x) = f(x) -\< f'(x_y) ,x \> $.} Wasserstein distance bound follows similarly. The full proofs are in \S\ref{sec:main_result_proof}.

 Both step sizes have better dimension dependence than existing methods $\cO(1/d)$ for non-convex semi-smooth potential~\citep{Liang2022APA}.
 In addition, Theorem \ref{thm:rgo} can recover Lemma 5.5 in \citet{gopi2022private}. \fan{We further extend the results in Theorem \ref{thm:rgo},  \ref{thm:rgo_w2} to the $\chi^2$-divergence in \S \ref{sec:rgo_chi}.}
 
\section{Improved complexity bounds of proximal sampling for semi-smooth potential}\label{sec:convergence}
Our overall algorithm uses the RGO implementation in Algorithm \ref{alg:rgo} to implement line \ref{line:sample_oracle} in the proximal sampler (Algorithm \ref{algo:AlternatingSampler}).
Combining the theoretical properties of our RGO implementation in \S \ref{sec:exact_rgo}, and the existing convergence results for the exact proximal sampler in Theorem \ref{thm:proximal_convergence},
we establish superior sampling complexity bounds under various conditions. 

Denote the distributions of the iterations $y_t$ and $x_t$ of the ideal proximal sampler by $\psi_t$ and $\mu_t$ respectively. Our RGO implementation is not exact, rendering different distributions along the iterations of the proximal sampler, denoted by $y_t \sim \hat \psi_t $ and $x_t \sim \hat \mu_t $. To establish the final complexity bounds, we need to quantify the difference between $\hat \psi_t$ (or $\hat \mu_t$) and $\psi_t$ (or $\mu_t$) caused by the inexact RGO.
We first present the following lemmas to control the accumulated error of the inexact RGO. Lemma \ref{lem:tv_telecope} is also informally mentioned in \citet[\S A]{lee2021structured}. Lemma \ref{lem:w2_telecope} is proved based on formulating the RGO as a backward diffusion, and then adopting a coupling argument, which is introduced in \citet[\S A]{chen2022improved}. The detailed proofs are in Appendix \ref{sec:lemmas}.

\begin{lemma}\label{lem:tv_telecope}
 Assume the output of Algorithm \ref{alg:rgo} follows 
 $ \hat \pi^{X|Y} (\cdot |y)$ that can achieve $$\left\| \hat \pi^{X|Y} (\cdot|y) -  \pi^{X|Y} (\cdot|y)  \right\|_\TV \le \zeta $$ for $~\forall ~y,$ then
 $\| \hat \mu_T - \mu_T \|_\TV \le \zeta T. $ 
\end{lemma}

\begin{lemma}\label{lem:w2_telecope}
 Assume the output of Algorithm \ref{alg:rgo} follows $ \hat \pi^{X|Y} (\cdot|y)$ that can achieve $$W_2 \left( \hat \pi^{X|Y} (\cdot|y) ,  \pi^{X|Y} (\cdot|y)  \right) \le \zeta $$ for $~\forall ~y$, and suppose the target distribution $\nu \propto \exp(-f)$ is log-concave, then $W_2( \hat \mu_T , \mu_T ) \le \zeta T. $
\end{lemma}

We next establish the complexity bounds of our algorithm in three settings, with strongly convex, convex, and non-convex potentials. The idea of proof is simple. Assume we iterate $T$ times in Algorithm \ref{algo:AlternatingSampler}. If the desired final error is $\delta$,
then we choose the RGO accuracy $\zeta = \Theta(\delta / T)$ by the above lemmas to ensure the accumulated error is small. Plugging the corresponding step size $\eta$ by Theorem \ref{thm:rgo} or \ref{thm:rgo_w2} into Theorem \ref{thm:proximal_convergence} gives the final results. The detailed proofs are in  \S\ref{sec:main_result_proof}.
\fan{Although our results in this section are with respect to TV / $W_2$, we extend them to the $\chi^2$-divergence setting in \S \ref{sec:convegence_chi}.}
\subsection{Strongly convex and smooth potential}

\begin{proposition}\label{prop:strongly_convex}
Suppose $f$ is $\beta$-strongly convex and $L_1$-smooth. Let $ \delta \in (0,1), \eta =
\tilde \cO \left( {1}/{ (L_1 \sqrt{d}) }  \right) 
$.
Then Algorithm \ref{algo:AlternatingSampler}, with Algorithm \ref{alg:rgo} as RGO step and initialization $x_0\sim \mu_0$, can find a random point $x_T$ that has $\delta$ total variation distance to the distribution $\nu \propto \exp(-f(x)) $ in $$T = \cO  \left( \frac{ L_1 \sqrt{d}}{ \beta }  \log \left( \frac{L_1 \sqrt{d}}{ \beta \delta } \right)  \log \left( \frac{ \sqrt{H_\nu (\mu_0)}}{ \delta} \right) \right) $$ steps. And we can find 
$x_T$ 
that has $\delta$ Wasserstein-2 distance to the distribution $\nu  $ in $$
T= \cO \left( \frac{ L_1 \sqrt{d}}{\beta} \log\left( \frac{L_1 d }{\beta \delta} \right) {\log \left( \frac{1}{ \delta } \sqrt{ \frac{ H_\nu (\mu_0)}{ \beta} } \right)} \right) 
$$ steps.
Furthermore,  each step accesses only $\cO(1)$ many $f(x)$ queries in expectation.
\end{proposition}

In this most classical setting, we compare our results with others in the literature.
Denote $\kappa: = L_1 / \beta $ as the condition number. By 
Lemma \ref{lem:comparison} and \citet[\S A]{chewi2021analysis}, we assume $H_\nu (\mu_0) = \cO(d)$. Our total complexity becomes 
$\tilde\cO(\kappa  \sqrt{d})$ 
for both total variance and Wasserstein distance, which is better than 
$\tilde \cO(\kappa d )$ 
in \citet[Corollary 7]{chen2022improved} also for the proximal sampler. Considering other sampling methods, our result $\tilde \cO(\kappa \sqrt{d})$ surpasses most of the existing bounds, including randomized midpoint Unadjusted Langevin Monte Carlo (LMC)~\citep{he2020ergodicity}, ULMC~\citep{cheng2018underdamped,dalalyan2020sampling,ganesh2020faster}, MALA with a warm start~\citep{chewi2021optimal,wu2021minimax}. 
In particular, \citet{wu2021minimax} shows $\tilde \Omega(\kappa\sqrt{d})$ is the lower bound for MALA to mix.
\citet{shen2019randomized} can achieve 
$\tilde \cO  \left(
    \kappa^{7/6} ( \frac{2}{\delta} \sqrt{\frac{d}{\beta} } )^{1/3}
    +
    \kappa ( \frac{2}{\delta } \sqrt{\frac{d}{\beta} } )^{2/3}
    \right) $ in terms of Wasserstein distance. Their bound is better in dimension dependence but depends 
    polynomially on $\delta$, and is therefore not a high-accuracy guarantee.

\subsection{Convex and semi-smooth potential}
\begin{proposition}\label{prop:convex}
Suppose $f$ is convex and $L_\alpha$-$\alpha$-semi-smooth. Let $ \delta \in (0,1), \eta =
\tilde \cO \left({1}/{ (L_\alpha^{\frac{2}{\alpha+1}}d^{\frac{\alpha}{\alpha+1}} ) } \right)$. 
Then we
can find a random point $x_T$ that has $\delta$ total variation distance to $\nu 
\propto \exp(-f(x))
$ in $$
T= \cO \left( \frac{W_2^2(\mu_0, \nu) L_\alpha^{ \frac{2}{\alpha+1} } d^{\frac{\alpha}{\alpha+1}}
}{\delta^2 } 
\log \left( \frac{W_2^2(\mu_0, \nu) L_\alpha^{ \frac{2}{\alpha+1} } d^{\frac{\alpha}{\alpha+1}}}{\delta^3 }\right)
\right)
$$ steps.
Furthermore,  each step accesses only $\cO(1)$ many $f(x)$ queries in expectation.
\end{proposition}
Assume $W_2^2(\mu_0, \nu) = \cO(\cM_2)$ where $\cM_2$ is the second moment of $\nu$. Then our bound becomes $\tilde \cO(\cM_2 L_\alpha^{ \frac{2}{\alpha+1} } d^{\frac{\alpha}{\alpha+1}} / \delta^2 )$. When $f$ is smooth and $\alpha=1$, our result $\tilde \cO(\cM_2 L_1 \sqrt{d} /\delta^2 )$ improves the bound $\tilde \cO(\cM_2 L_1 d / \delta^2 )$ in \citet[Corollary 6]{chen2022improved}, which is the state-of-art (in dimension) complexity  for log-concave smooth sampling. 
When $f$ is Lipschitz, our result $\tilde \cO(\cM_2 L_0^2 /\delta^2 )$  improves the bound $\tilde \cO(\cM_2 L_0^2 d /\delta^4 )$ in \citet[Theorem 1]{lehec2021langevin}.
When $f$ is semi-smooth, we also improve the bound $\tilde \cO( \sqrt{\cM_4} L_\alpha^{\frac{2}{\alpha+1}} d /\delta )$ in \citet[Theorem 4.2]{liang2022proximal} in terms of dimension. Here $\cM_4$ is the fourth central moment of $\nu$.
\begin{remark}
An alternative approach to sample from a log-concave distribution is to first construct a regularized strongly convex potential $\hat f(x):= f(x) + w\| x - x^*\|^2/2$ with $x^*$ being an (approximated) minimizer of $f$. Following \citet{liang2021proximal}, Algorithm \ref{alg:rgo} can be modified so that the proximal sampler can sample from $\exp(-\hat f)$ with complexity $\tilde \cO(L_\alpha^{ \frac{2}{\alpha+1} } d^{\frac{\alpha}{\alpha+1}}/w)$. With a proper choice of $w$, we arrive at an algorithm to sample from $\nu\propto \exp(-f)$ with complexity $\tilde \cO(\sqrt{\cM_4}L_\alpha^{ \frac{2}{\alpha+1} } d^{\frac{\alpha}{\alpha+1}}/\delta)$.
\end{remark}

\subsection{Sampling from non-log-concave distributions \fan{satisfying isoperimetric inequalities} }
Since \eqref{eq:lsi} implies that the distribution has a sub-Gaussian tail, we only present the result for smooth potentials ($\alpha=1$) when LSI is satisfied.

\begin{proposition}\label{prop:lsi}
Suppose $\nu \propto \exp(-f)$ satisfies $C_\LSI$-\ref{eq:lsi} and $f$ is $L_1$-smooth. Let $ \delta \in (0,1), \eta =
\tilde \cO \left( {1}/{ (L_1 \sqrt{d}) }  \right). $ 
Then we
can find a random point $x_T$ that has $\delta$ total variation distance to $\nu$
in $$
T =   \cO  \left( \frac{ L_1 \sqrt{d}}{ C_\LSI }  \log \left( \frac{L_1 \sqrt{d} }{ C_\LSI \delta} \right)  \log \left( \frac{ \sqrt{H_\nu (\mu_0)}}{ \delta} \right) \right)
$$ steps.
Furthermore,  each step accesses only $\cO(1)$ many $f(x)$ queries in expectation.
\end{proposition}

We can also define a ``condition number'' $\hat \kappa = L_1 / C_\LSI $, and assume $H_\nu( \mu_0) = \cO(d)$. Then our result becomes $\tilde\cO  \left( { \hat \kappa \sqrt{d}}  
\right)$, whereas \citet[Corollary 7]{chen2022improved}
and \citet[Theorem 3.1]{Liang2022APA} give $\tilde \cO(\hat \kappa d 
)$. 
Our bound is also better than the order $\tilde \cO( \hat \kappa^2 d /\delta^2 )$ in \citet[Theorem 7]{chewi2021analysis,erdogdu2021convergence}.

\begin{proposition}\label{prop:pi}
Suppose $\nu  \propto \exp(-f)$ satisfies $C_\PI$-\ref{eq:pi} and $f$ is $L_\alpha$-$\alpha$-semi-smooth.
Let $ \delta \in (0,1), \eta =
\tilde \cO \left({1}/{ (L_\alpha^{\frac{2}{\alpha+1}}d^{\frac{\alpha}{\alpha+1}} ) } \right)$. 
Then we
can find a random point $x_T$ that has $\delta$ total variation distance to $\nu 
$ in $$
T= \cO \left( 
 \frac{L_\alpha^{\frac{2}{\alpha+1}}d^{\frac{\alpha}{\alpha+1}} }{C_\PI}\log \left(\frac{L_\alpha^{\frac{2}{\alpha+1}} d^{\frac{\alpha}{\alpha+1}} 
}{C_\PI \delta } \right) \log \left( \frac{\chi_\nu^2(\mu_0)}{\delta^2} \right) \right)
$$ steps.
Furthermore,  each step accesses only $\cO(1)$ many $f(x)$ queries in expectation.
\end{proposition}

By 
Lemma \ref{lem:comparison} and \citet[\S A]{chewi2021analysis}, we assume $ \chi_\nu^2(\mu_0) = \cO(\exp( d))$. 
Then our result becomes $\tilde\cO \left( 
 \frac{L_\alpha^{\frac{2}{\alpha+1}}d^{\frac{2\alpha+1}{\alpha+1}} }{C_\PI}
\right)$. This improves the result $\tilde \cO \left( \frac{L_\alpha^{\frac{2}{\alpha+1}}d^2 }{C_\PI} 
\right)$ in \citet{Liang2022APA} and $\tilde \cO \left( \frac{L_\alpha^{{2}/{\alpha}}d^{2 + 1/ \alpha} }{C_\PI^{1+1/\alpha}  \delta^{2/ \alpha} }  \right) $ in \citet{chewi2021analysis}.

\section{Proximal sampling for composite potentials}\label{sec:composite}

In this section, we consider the composite potential $f= \sum_{j=1}^n f_j $ that satisfies 
\begin{align}\label{eq:composite}
\|f_j'(u) - f_j'(v)  \| \le 
L_{\alpha_j} \| u -v\|^{\alpha_j} , ~~ \forall~ u , v \in \mR^d  ~~, \forall~ 1\le j \le n  
\end{align}
with $\alpha_j \in [0, 1]$ for all $j$.
When $n=1$, this can recover the assumption \eqref{eq:semi-smooth}.
When $n=2$, it can also recover the popular ``smooth+non-smooth'' function assumption.
In this section we extend the analysis in the previous sections for semi-smooth $f$ to composite $f$.

\subsection{RGO complexity}

Again, to simplify the argument, we assume that the stationary point of 
$f_y^\eta$ can be computed exactly. Otherwise,
we can adopt the same argument in \S \ref{sec:approximate_opt} to obtain the same complexity order.

Since $f$ is composite, we naturally let $g= \sum_{j=1}^n g_j $, where $g_j(x):= f_j(x) - \< f_j'(x_y) , x \> $. 
Clearly, $x_y$ is the stationary point of all $g_j$, namely, $g_j'(x_y)=0$.
It is easy to check that Lemma \ref{lem:equi_rgo} and \ref{lem:dist_relation} still hold for the composite potential. Thus, the RGO step aims to sample from $\exp(-g(x)-\nicefrac{1}{2\eta}\|x-y\|^2)$ for a fixed $y$. Next, for the complexity analysis,
the crux is also to develop a concentration inequality for composite function 
$g$. The proof is naturally based on the probabilistic uniform bound.

\begin{corollary}[Gaussian concentration inequality for composite functions] \label{cor:composite_concentration}
Let $X \sim \cN(m,\eta\ide)$ be a Gaussian random variable in $\Rd$, and let $\ell$ be a composite function satisfying \eqref{eq:composite}. Assume $ \ell'_j (m) = 0$ ~for~ $\forall ~1\le j \le n$. Then for any $r>0, ~0 \le {\alpha_j} \le 1$ and $\sum_{j=1}^n w_j =1, w_j\ge 0$, one has 
\begin{align} \label{eq:composite_conc}
\Pr( \ell(X)-\mE(\ell(X))\geq r) 
\leq \left(1-\frac{\epsilon}{d} \right)^{-\frac{d}{2}}
\sum_{i=j}^n 
 \exp\left(-\frac{C_j \epsilon^{\frac{{\alpha_j}}{1+{\alpha_j}}}(w_j r)^{\frac{2}{1+{\alpha_j}}}}{L_{\alpha_j}^{\frac{2}{1+{\alpha_j}}}d^{\frac{{\alpha_j}}{1+{\alpha_j}}}\eta}\right), ~~ \forall \epsilon \in (0,d), 
\end{align}
where
$C_j = \Ciconst$.
\end{corollary}

With this concentration inequality,
in view of the fact that $g_j'(x_y)=0$, we obtain similar RGO complexity results.
Our step size in the following theorems can recover the $n=1$ case in \S \ref{sec:exact_rgo}.

 \begin{theorem}[RGO complexity in total variation]\label{thm:composite_rgo_tv}
Assume $f$ satisfies \eqref{eq:composite}. If the step size 
$$\eta \le 
\left( \Consteta  \left( \sum_{j=1}^n L_{\alpha_j}^{\frac{1}{\alpha_j+1}} d^{ \frac{\alpha_j}{ 2(\alpha_j+1)} } \right)^2(1+ \log (1+ 12 n /\zeta ))
\right)^{-1}
,$$
then for any $\zeta>0$, Algorithm \ref{alg:rgo} returns a random point $x$ that has $\zeta$ total variation distance to the distribution proportional to $
\pi^{X|Y}(\cdot | y)
$. Furthermore, if $~0 < \zeta < 1$, then the algorithm access $\cO(1)$ many $f(x)$ in expectation.
 \end{theorem}
 \begin{theorem}[RGO complexity in Wasserstein distance]\label{thm:composite_rgo_w2}
Assume $f$ satisfies \eqref{eq:composite}. If the step size 
\begin{align}
\eta \le \min \left(
\left( \Consteta  \left( \sum_{j=1}^n L_{\alpha_j}^{\frac{1}{\alpha_j+1}} d^{ \frac{\alpha_j}{ 2(\alpha_j+1)} } \right)^2 (2+ \log( 1+ {192 n  (d^2 + 2d)}/{ \zeta^4} )) \right)^{-1} , ~~~
1\right)  ,  
\end{align}
then Algorithm \ref{alg:rgo} returns a random point $x$ that has $\zeta$ Wasserstein-2 distance to the distribution $\pi^{X|Y}(\cdot | y)$. If 
 ~$0< \zeta < 2\sqrt{2d} $, then the algorithm access $\cO(1)$ many $f(x)$ in expectation.
 \end{theorem}

\subsection{Complexity bounds of proximal sampling}

Next we provide the total complexity under several conditions for composite potential. 
As in \S \ref{sec:convergence}, the results in this section are obtained by plugging the RGO step size (Theorem \ref{thm:composite_rgo_tv}, \ref{thm:composite_rgo_w2}) into proximal sampler convergence results in Theorem \ref{thm:proximal_convergence}.
Due to the high similarity to \S \ref{sec:convergence}, we omit the proofs for the results in this section. Throughout this section, we denote $M_{L,d} :=  \left( \sum_{j=1}^n L_{\alpha_j}^{\frac{1}{\alpha_j+1}} d^{ \frac{\alpha_j}{ 2(\alpha_j+1)} } \right)^2 $ and assume that $\eta 
= \tilde \cO (1/M_{L, d}) $.

\begin{proposition}
Suppose $f$ is $\beta$-strongly convex and satisfies \eqref{eq:composite} and $ \delta \in (0,1). 
$
Then Algorithm \ref{algo:AlternatingSampler},  with Algorithm \ref{alg:rgo} as RGO step, can find a random point $x_T$ that has $\delta$ total variation distance to the distribution $\nu \propto \exp(-f(x)) $ in $$T = \cO  \left( \frac{ M_{L,d } }{ \beta }  \log \left( \frac{M_{L,d } }{ \beta \delta } \right)  \log \left( \frac{ \sqrt{H_\nu (\mu_0)}}{ \delta} \right) \right) $$ steps. And we can find $x_T$ that has $\delta$ Wasserstein-2 distance to the distribution $\nu$ in $$
T= \cO \left( \frac{M_{L,d } }{\beta} \log\left( \frac{M_{L,d } \sqrt{d} }{\beta \delta} \right) {\log \left( \frac{1}{ \delta } \sqrt{ \frac{ H_\nu (\mu_0)}{ \beta} } \right)} \right) 
$$ steps.
Furthermore,  each step accesses only $\cO(1)$ many $f(x)$ queries in expectation.
\end{proposition}
We again assume $H_\nu (\mu_0) = \cO(d)$~\citep[\S A]{chewi2021analysis}. 
If $f = f_1 + f_2$, where $f_1$ is strongly convex and $L_1$-smooth, and $f_2$ is $L_0$-Lipschitz, then we have $\tilde \cO((L_0 + L_1^{1/2} d^{1/4})^2 / \beta)$, whereas ~\citet{bernton2018langevin} gives $\tilde \cO( L_0^2 d/ (\beta \delta^4) )$, \citet{salim2020primal} gives $\tilde \cO( (L_0^2 + L_1 d )/ (\beta^2 \delta^2 ) ) $, and
\citet[Theorem 5.5]{liang2022proximal} gives $\tilde \cO(  \max (L_0^2 ,L_1) ~ d /\beta )$.

\begin{proposition}\label{prop:composite}
Suppose we
can find a random point $x_T$ that has $\delta$ total variation distance to $\nu$ in $T$ steps.  Let $ \delta \in (0,1),$  and $\nu \propto \exp(-f)$, where $f$ satisfies \eqref{eq:composite}. \\
1) If $\nu$ is log-concave, then 
$T= \cO \left(  \frac{M_{L,d } W_2^2(\mu_0, \nu) 
}{\delta^2 } 
\log \left( \frac{ M_{L,d } W_2^2(\mu_0, \nu) }{\delta^3 }\right)
\right);$ \\
2)
If $\nu$ satisfies $C_\LSI$-\ref{eq:lsi},
then $
T =   \cO  \left( \frac{  M_{L,d } }{ C_\LSI }  \log \left( \frac{ 
 M_{L,d } }{ C_\LSI \delta} \right)  \log \left( \frac{ \sqrt{H_\nu (\mu_0)}}{ \delta} \right) \right) ;
$  \\
3) If $\nu$ satisfies $C_\PI$-\ref{eq:pi}, then 
$
T =  \cO \left( 
 \frac{ M_{L,d } }{C_\PI}\log \left(\frac{ M_{L,d }
}{C_\PI \delta } \right) \log \left( \frac{\chi_\nu^2(\mu_0)}{\delta^2} \right) \right) ;
$ \\
Furthermore,  each step accesses only $\cO(1)$ many $f(x)$ queries in expectation.
\end{proposition}
Consider log-concave $\nu$ and the potential $f$ is a "smooth + semi-smooth" function.
Our result becomes $\tilde \cO( \cM_2 (L_1^{1/2} d^{1/4} + L_\alpha^{1/(\alpha+1)} d^{\alpha/(2(\alpha+1))} )^2 / \delta^2 )$,
which improves $\tilde \cO( \sqrt{\cM_4} \max( L_1, L_\alpha^{2/(\alpha+1 )}) d /\delta  )$ in \citet[Theorem 5.4]{liang2022proximal} in terms of dimension. Our result for non-log-concave $\nu$  also improves the bounds in \citet{nguyen2021unadjusted,Liang2022APA} (see Table \ref{tab:t2}).

\section{Concluding remark}\label{sec:conclude}
We propose and analyze a novel RGO realization of the proximal sampler. The core of our analysis is a new Gaussian concentration inequality for semi-smooth functions, which is itself of independent interest.
With this concentration inequality, we significantly improve the dimension dependence of RGO.
We then analyze the accumulated error caused by our inexact RGO, which is combined with proximal sampler convergence results to give the total complexity. Our complexity bounds are better than almost all existing results in all the classical settings (strongly log-concave, log-concave, LSI, PI) as well as more general settings with semi-smooth potentials or composite potentials.

We leave a few directions for future study:
1) How to generalize our RGO algorithm to settings where the target potential $f$ is an empirical risk or population risk? This should be achievable by merging Algorithm \ref{alg:rgo} and \citet[Algorithm 2]{gopi2022private}. This will be useful for private optimization in differential privacy.
2) Is there any other RGO algorithm that has even better dimension dependence than Algorithm \ref{alg:rgo}? 
3) Our proof techniques make our concentration inequality only applicable to Gaussian distributions. Is there a similar concentration inequality as in Theorem \ref{thm:concentration} but for more general distributions satisfying LSI?

\acks{We thank Daogao Liu for useful discussions. We extend our gratitude to the anonymous reviewers for their
invaluable feedback that enhanced this manuscript. Financial support from NSF under grants 1942523, 2008513, and 2206576 is greatly acknowledged.}

\newpage
\bibliography{reference}
\bibliographystyle{icml_2022}

\newpage
\begin{appendix}

\tableofcontents
\newpage

\section{Proof of technical lemmas}
\label{sec:lemmas}
\subsection{Proof of Lemma \ref{lem:equi_rgo}}
\begin{proof}
Since $ f'(x_y) + \frac{1}{\eta } (x_y - y) = 0$, we have $x_y = y - \eta f'(x_y) .$ This implies
\begin{align}
g(x) + \frac{1}{2 \eta} \|x-x_y \|^2 
& = g(x) + \frac{1}{2 \eta} \|x- y + \eta f'(x_y) \|^2  \\
& = g(x) + \frac{1}{2 \eta} \|x- y \|^2 + 
 \< x,  f'(x_y) \>
-
\< y,  f'(x_y) \>
+ \frac{\eta}{2} \|f'(x_y)\|^2 \\
& =  g(x) + \frac{1}{2 \eta} \|x- y \|^2 + 
 \< x,  f'(x_y) \> + \text{constant} \\
 & =  f (x)+\frac{1}{2\eta}\|x-y\|^{2}+ \text{constant} .
\end{align}
We use $f(x) = g(x) + \< f'(x_y) , x \> $ in the last equality.
\end{proof}

\subsection{Proof of Lemma \ref{lem:dist_relation}}

\begin{proof}
Following the definition and Lemma \ref{lem:equi_rgo}, we have
\begin{align}
 \frac{\d \pi^{X|Y} }{ \d x} 
 & = \frac{ \exp(-f (x)-\frac{1}{2\eta}\|x-y \|^{2}) }{ \int \exp(-f (x)-\frac{1}{2\eta}\|x-y \|^{2}) \d x}  
 = \frac{ \exp(-g (x)-\frac{1}{2\eta}\|x-x_y \|^{2}) }{ \int \exp(-g (x)-\frac{1}{2\eta}\|x-x_y \|^{2}) \d x} \\
 & = \frac{ \exp(-g (x)) \exp(-\frac{1}{2\eta}\|x-x_y \|^{2}) / \int \exp(-\frac{1}{2\eta}\|x-x_y \|^{2}) \d x }{ \int \exp(-g (x)) (\exp(-\frac{1}{2\eta}\|x-x_y \|^{2}) /  \int \exp(-\frac{1}{2\eta}\|x-x_y \|^{2}) \d x ) \d x } \\ 
 & = \frac{\exp(-\frac{1}{2\eta}\|x-x_y \|^{2}) }{\int \exp(-\frac{1}{2\eta}\|x-x_y \|^{2}) \d x} \cdot \frac{ \exp(-g (x))  }{ \mE[ \exp(-g (x)) ] }  = \frac{\d \phi}{\d x} \cdot \frac{\exp(-g(x))}{ \mE_{x \sim \phi} \exp(-g(x)) }.
\end{align}
Next, since
\begin{align}
\mE[\rho | x] 
& = \mE[\exp( g(z) - g(x) ) | x] 
= \mE[\exp( g(z) )]  \exp(- g(x) ), \\
\mE[\rho] 
& = \mE[\exp( g(z) )]  \mE[ \exp(- g(x) )],
\end{align}
we get ${\exp(-g(x))}/{ \mE \exp(-g(x)) } = \mE[\rho | x]  / \mE[\rho] . $ Finally, since Algorithm \ref{alg:rgo} is rejection sampling, 
\begin{align}
 \frac{\d \hat \pi^{X|Y} }{ \d x} =     \frac{\d \phi}{\d x} \cdot \frac{ \Pr( u \le \frac{1}{2} \rho| x) }{ \Pr( u \le \frac{1}{2} \rho) } .
\end{align}
By tower property and the fact that $u$ follows the uniform distribution over $[0.\,1]$, the acceptance probability is 
$$\Pr( u \le \frac{1}{2 } \rho ) = \mE [\Pr( u \le \frac{1}{2} \rho |  \rho )  ]  = \frac{1}{2} \mE[\bar \rho] .$$ Similarly, $\Pr(u \le \frac{1}{2 } \rho | x) = \frac{1}{2} \mE[\bar \rho | x]$. These conclude that $ \frac{\d \hat \pi^{X|Y} }{ \d x} =     \frac{\d \phi}{\d x} \cdot \frac{\mE[\bar \rho | x] }{\mE[\bar \rho ] } .$
\end{proof}

\subsection{Proof of Lemma \ref{lem:tv_telecope}}
\begin{proof}
 Indeed, by triangular inequality and Jensen inequality, 
    \begin{align}
& \|\hat \mu_t(x) - \mu_t(x) \|_\TV  
= \left\| \int  \hat \psi_t(y) \hat \pi^{X|Y} (x |y) \d y - \int  \psi_t(y)  \pi^{X|Y} (x |y) \d y\right\|_\TV \\
\le & \left\| \int  \hat \psi_t(y) \left(\hat \pi^{X|Y} (x |y) -  \pi^{X|Y} (x |y) \right) \d y \right\|_\TV  + \left\| \int  (\hat \psi_t(y) - \psi_t(y) ) \pi^{X|Y} (x |y) \d y \right\|_\TV \\
\le &  \int  \hat \psi_t(y) \left\| \hat \pi^{X|Y} (x |y) -  \pi^{X|Y} (x |y)  \right\|_\TV \d y   + \int  \left(\left | \hat \psi_t(y) - \psi_t(y) \right |\int \pi^{X|Y}(x|y) \d x \right) \d y\\
\le & ~ \zeta + \left\| \hat \psi_t(y) - \psi_t(y) \right\|_\TV .
\end{align}
Moreover, denote $G$ as the density of the distribution $\cN(0,\eta \ide)$,
\begin{align}
\left\| \hat \psi_{t}(y) - \psi_{t}(y) \right\|_\TV 
=\left\| \int (\hat \mu_{t-1}(x) - \mu_{t-1} (x) ) G(y-x) \d x\right\|_\TV
\le \|\hat \mu_{t-1} - \mu_{t-1} \|_\TV.
\end{align}
Thus, $\|\hat \mu_t - \mu_t \|_\TV  \le   \zeta + \|\hat \mu_{t-1} - \mu_{t-1} \|_\TV $. Finally, because $\hat \psi_1 = \psi_1$, there is
$\|\hat \mu_{T} - \mu_{T} \|_\TV  \le   \zeta T .$
\end{proof}

\subsection{Proof of Lemma \ref{lem:w2_telecope}}

\begin{proof}
Denote the marginal distribution of $Y$ in the $t$-th iteration of the idea proximal sampler and approximate proximal sampler by $\psi_t, \hat\psi_t$ respectively. It is well known the Gaussian convolution is contractive with respect to the Wasserstein-2 metric. Thus
    \begin{align}\label{eq:w2_contra_1}
          W_2 (\psi_t, \hat\psi_t) \le W_2 (\mu_{t-1}, \hat \mu_{t-1}).     
    \end{align}
Following the Doob's $h$-transform, the RGO $\pi^{X|Y}$ can be realized by simulating the backward diffusion (see \cite[\S A]{chen2022improved}) 
    \[
        \d Z_s = -\nabla \log \pi_s (Z_s) \d t + \d B_s, 
    \]
over the time interval $[0,\, \eta]$, where $B_s$ is a standard Wiener process and $\pi_s = \nu * \cN(0, s \ide)$.
Let $Z_s, \hat Z_s$ be two copies of the above process initialized at $Z_\eta \sim \psi_t, \hat Z_\eta \sim \hat\psi_t$ respectively. We then use a coupling argument to show the RGO for log-concave distribution is also contractive with respect to $W_2$. In particular, let $Z_s, \hat Z_s$ be driven by a common Wiener process $B_s$ and $Z_\eta, \hat Z_\eta$ be coupled in such a way that $\mE \|Z_\eta - \hat Z_\eta\|^2 = W_2^2 (\psi_t, \hat \psi_t)$, then
    \[
        \frac{\d}{\d t}\|Z_s- \hat Z_s\|^2 = 2\langle - \nabla \log \pi_s (Z_s)+ \nabla \log \pi_s (\hat Z_s), Z_s- \hat Z_s\rangle \ge 0.
    \]
The last inequality is due to the fact that $\nu$ and thus $\pi_s$ is log-concave. It follows that 
    \begin{align}
        & W_2^2\left(\int \pi^{X|Y} (x |y)\psi_t(y) \d y, \int \pi^{X|Y} (x |y)\hat \psi_t(y) \d y \right) \\
        \le & \mE \|Z_0-\hat Z_0\|^2 \le \mE \|Z_\eta-\hat Z_\eta \|^2 = W_2^2 (\psi_t, \hat \psi_t). \label{eq:w2_contra_2}
   \end{align}
Moreover, by definition of Wasserstein distance,  we have
\begin{align}\label{eq:w2_contra_3}
W_2\left( \int  \hat \pi^{X|Y} (x |y) \hat \psi_t(y) \d y , \int \pi^{X|Y} (x |y) \hat \psi_t(y) \d y \right) 
\le \sqrt{\int \hat \psi_t(y) W_2^2( \hat \pi^{X|Y} , \pi^{X|Y} ) \d y} \le \zeta. ~~~~
\end{align}
By triangular inequality, we have
\begin{align}
& W_2( \hat \mu_t, \mu_t )  
= W_2\left( \int   \hat \pi^{X|Y} (x |y) \hat \psi_t(y) \d y , \int   \pi^{X|Y} (x |y) \psi_t(y) 
 \d y\right) \\
{\le} ~ & W_2\left( \int  \hat \pi^{X|Y} (x |y)  \hat \psi_t(y) \d y , \int   \pi^{X|Y} (x |y) \hat \psi_t(y)  \d y\right) \\
& +
W_2\left( \int  \pi^{X|Y} (x |y)  \hat \psi_t(y)  \d y , \int   \pi^{X|Y}  \psi_t(y) (x |y) \d y\right) \\
\overset{\eqref{eq:w2_contra_2}, \eqref{eq:w2_contra_3}}{\le} & \zeta + W_2^2 (\hat \psi_t,  \psi_t) \overset{\eqref{eq:w2_contra_1}}{\le} \zeta + W_2 (\hat \mu_{t-1},  \mu_{t-1}).
\end{align}
Finally, because $\hat \psi_1 = \psi_1$, there is
$W_2( \hat \mu_T , \mu_T ) \le \zeta T. $
\end{proof}

\subsection{Supportive lemmas}
In this section, we list the lemmas pertaining to our analysis.
\begin{lemma}\label{lem:comparison}
Under \ref{eq:pi}, \citet{liu2020poincare} together with standard comparison inequalities implies that 
\begin{align}
\max \left\{ \frac{\|\mu - \nu \|_\TV^2}{2} , \log \left(1+ \frac{C_\PI}{2} W_2^2(\mu, \nu) \right) , H_\nu(\mu) \right\}  \le  R_{2, \nu} (\mu),
\end{align}    
where $R_{2,\nu} (\mu):= \log \int (\mu^2/ \nu) $ ~is \Ren divergence of order 2.
\end{lemma}
\begin{lemma}[Proposition 7.10 in \citet{villani2021topics}]\label{lem:tv_w2}
    Let $\mu$ and $\nu$ be two probability measures on $\mR^d$. Then for any $ m \in \mR^d$, 
    \begin{align}
    W_2^2(\mu, \nu) 
    \le 2 \int \|m  - x\|_2^2 \d |\mu -\nu|(x)
    = 2 \left\| \|m  - \cdot\|_2^2  (\mu -\nu ) \right\|_\TV.
    \end{align}
\end{lemma}

\section{Gaussian concentration inequality}
\subsection{Proof of Theorem \ref{thm:concentration}}\label{sec:conc_proof}
\begin{proof}
 Without loss of generality, we assume $m = 0$. To see this, define $g(x) \coloneqq \ell(x+m)$ which is also an $L_{\alpha}$-$\alpha$-semi-smooth function and satisfy $ g'(0) = 0$, and notice that $\mE(\ell(X)) = \mE(g(Y))$ where  $Y \sim \cN(0,\eta\ide)$. It follows that
 \begin{align}
 \Pr(\ell(X)-\mE(\ell(X)) \geq r)
&= (2\pi\eta)^{-d/2}\int_{\Rd} \exp\left(-\frac{1}{2\eta}\|x-m\|_2^2\right)\mathbbm{1}_{\ell(x)-\mE(\ell(x)) \geq r}\d x \\
& = (2\pi\eta)^{-d/2}\int_{\Rd} \exp\left(-\frac{1}{2\eta}\|y\|_2^2\right)\mathbbm{1}_{g(y)-\mE(g(y)) \geq r} \d y\\
& = \Pr(g(Y)-\mE(g(Y)) \geq r).
\end{align}
Hence in what follows, $m = 0$. The following proof is based on the elegant argument of Maurey and Pisier \citep[Theorem 2.1]{pisier2006probabilistic}.
Let $G$ and $H$ be two independent Gaussian variables following $\cN(0,\eta\ide)$, then for any $\lambda >0$, by independence and Jensen's inequality, one has 

\begin{align}
\mE \exp (\lambda (\ell(G)-\ell(H)) ) 
&= \mE\exp (\lambda (\ell(G)-\mE \ell(G)) )\mE\exp (\lambda (\mE \ell(H)-\ell(H)) ) \\
& \geq \mE\exp (\lambda (\ell(G)-\mE \ell(G)) )\exp \mE (\lambda (\mE \ell(H)-\ell(H)) ) \\
& = \mE\exp (\lambda (\ell(G)-\mE \ell(G)) ).
\end{align}
Let $G_{\theta} \coloneqq G \sin\theta+ H\cos\theta$, $\theta \in [0,\pi/2]$. Using the fundamental theorem of calculus along $\theta$, one obtains
\begin{align}
\ell(G)-\ell(H) = \int_0^{\frac{\pi}{2}} \nabla \ell(G_{\theta}) \cdot (G\cos\theta-H\sin\theta)\d \theta.
\end{align}
It follows that
\begin{align}
\mE \exp (\lambda (\ell(G)-\ell(H)) ) 
& = \mE\exp\left(\lambda \int_0^{\frac{\pi}{2}} \nabla \ell(G_{\theta}) \cdot (G\cos\theta-H\sin\theta)\d\theta \right) \\
& \leq \frac{2}{\pi}\mE \int_0^{\frac{\pi}{2}}\exp\left(\frac{\pi}{2}\lambda  \nabla \ell(G_{\theta}) \cdot (G\cos\theta-H\sin\theta) \right)\d\theta \\
& =  \frac{2}{\pi} \int_0^{\frac{\pi}{2}} \mE\exp\left(\frac{\pi}{2}\lambda  \nabla \ell(G_{\theta}) \cdot (G\cos\theta-H\sin\theta) \right)\d\theta.
\end{align}
Since $G_{\theta}$ and $G\cos\theta-H\sin\theta$ are two independent Gaussian variables, $\nabla \ell(G_{\theta}) \cdot (G\cos\theta-H\sin\theta)$ is equidistributed as $Z\|\nabla \ell(G_{\theta})\|_2 $ where $Z \sim \cN(0,\eta)$ for any $\theta$. This implies
\begin{align}
\mE\exp\left(\frac{\pi}{2}\lambda  \nabla \ell(G_{\theta}) \cdot (G\cos\theta-H\sin\theta)\right) 
&= \mE\exp\left(\frac{\pi}{2}\lambda  Z\|\nabla \ell(G_{\theta})\|_2\right)\\
&= \mE_{G_{\theta}}\mE_Z\left(\exp\left(\frac{\pi}{2}\lambda  Z\|\nabla \ell(G_{\theta})\|_2 \right) \mid G_{\theta}\right) \\
&= \mE_{G_{\theta}}\exp\left(\frac{\pi^2}{8}\eta\lambda^2\|\nabla \ell(G_{\theta})\|_2^2\right).
\end{align}
As the distribution of $G_{\theta}$ does not depend on $\theta$, in the rest of the proof, we drop $\theta$. Then,
\begin{align}\label{eq:exp_grad}
\mE\exp (\lambda (\ell(G)-\mE \ell(G)) ) \leq \mE \exp (\lambda (\ell(G)-\ell(H)) )  \leq  \mE_{G}\exp\left(\frac{\pi^2}{8}\eta\lambda^2\|\nabla \ell(G)\|_2^2\right).
\end{align}
Combining with  Markov inequality yields 
\begin{align}
\Pr(\ell(X)-\mE(\ell(X))\geq r) 
& \leq \inf_{\lambda > 0} \frac{\mE \exp(\lambda(\ell(X)-\mE(\ell(X))))}{\exp(\lambda r)} \\
& \leq  \inf_{\lambda >0} \frac{\mE_{G}\exp(\frac{\pi^2}{8}\eta\lambda^2\|\nabla \ell(G)\|_2^2)}{\exp(\lambda r)}.
\end{align}
To study the properties of the optimization problem $\inf_{\lambda >0} \frac{\mE_{G}\exp(\frac{\pi^2}{8}\eta\lambda^2\|\nabla \ell(G)\|_2^2)}{\exp(\lambda r)}$, we consider three cases based on different $\alpha$.
\begin{enumerate}
\item $\alpha = 0$

In this case, $\ell$ is an $L_0$-Lipschitz function, which means $\|\nabla \ell(G)\|_2^2 \leq L_0^2$. Hence, our result coincides with the classical Gaussian concentration inequality for Lipschitz functions. Taking $\lambda = \frac{4r}{\pi^2L_0^2\eta }$ yields 
\begin{align}\label{eq:alpha0_conc}
\Pr(\ell(X)-\mE(\ell(X))\geq r) \leq  \exp(-\frac{2}{\pi^2}\frac{r^2}{L_0^2\eta }).
\end{align}
\item $\alpha = 1$

In this case, one could simplify the optimization function in an explicit way. Assume $-\frac{1}{2\eta} + \frac{\pi^2}{8}L_1^2\eta\lambda^2 <0$ (the condition that ensures $\mE_G$ is finite), then
\begin{align}
\frac{\mE_{G}\exp(\frac{\pi^2}{8}\eta\lambda^2\|\nabla \ell(G)\|_2^2)}{\exp(\lambda r)} 
& \leq \frac{\mE_{G}\exp(\frac{\pi^2}{8} L_1^2\eta\lambda^2\|G\|_2^2)}{\exp(\lambda r)} \\
& = \left(\frac{1}{1-\frac{\pi^2}{4}L_1^2\eta^2\lambda^2}\right)^{d/2}\exp(-\lambda r).
\end{align}
Let $\lambda = \frac{k}{L_1^2d\eta^2}$ where $k$ is a positive parameter. Here $k$ is chosen such that the condition  $-\frac{1}{2\eta} + \frac{\pi^2}{8}L_1^2\eta\lambda^2 <0$ holds. Then,
\begin{align}
\Pr(\ell(X)-\mE(\ell(X))\geq r) \leq  (1-\frac{\pi^2k^2}{4L_1^2d^2\eta^2})^{-\frac{d}{2}}\exp(-\frac{kr}{L_1^2d\eta^2}).
\end{align}
In the last step, let $\frac{\pi^2k^2}{4L_1^2d\eta^2}=\epsilon$, leading to
\begin{equation}\label{eq:alpha1_conc}
\Pr(\ell(X)-\mE(\ell(X))\geq r) \leq  (1-\frac{\epsilon}{d})^{-\frac{d}{2}}\exp(-\sqrt{\frac{4\epsilon}{\pi^2}}\frac{r}{L_1d^{1/2}\eta }), ~~ \forall \epsilon \in (0,d).
\end{equation}
Note that given the value of $\epsilon$, $(1-\frac{\epsilon}{d})^{-\frac{d}{2}}$ is bounded for $d$.

\item $0 <\alpha <1$

By Young's inequality, for any $\omega>0$, one obtains
$\|G\|^{2\alpha}_2 \leq \alpha\|G\|^2_2/\omega + (1-\alpha)\omega^{\frac{\alpha}{1-\alpha}} $.
Hence, with the assumption $1-\frac{\pi^2}{4}L_{\alpha}^2\eta^2\lambda^2\frac{\alpha}{\omega}>0$,
\begin{align}
\frac{\mE_{G}\exp(\frac{\pi^2}{8}\eta\lambda^2\|\nabla \ell (G)\|_2^2)}{\exp(\lambda r)} 
& \leq \frac{\mE_{G}\exp(\frac{\pi^2}{8} L_{\alpha}^2\eta\lambda^2\|G\|_2^{2\alpha})}{\exp(\lambda r)} \\
& \leq \frac{\mE_{G}\exp\left(\frac{\pi^2}{8} L_{\alpha}^2\eta\lambda^2(\alpha\|G\|^2/\omega + (1-\alpha)\omega^{\frac{\alpha}{1-\alpha}})\right)}{\exp(\lambda r)} \\
& = (1-\frac{\pi^2}{4}L_\alpha^2\eta^2\lambda^2\frac{\alpha}{\omega})^{-\frac{d}{2}}\exp(\frac{\pi^2}{8}L_\alpha^2\eta\lambda^2(1-\alpha)\omega^{\frac{\alpha}{1-\alpha}})\exp(-\lambda r). \label{eq:young_split}
\end{align}
Denote $F(\lambda, \omega) \coloneqq  (1-\frac{\pi^2}{4}L_\alpha^2\eta^2\lambda^2\frac{\alpha}{\omega})^{-\frac{d}{2}}\exp(\frac{\pi^2}{8}L_\alpha^2\eta\lambda^2(1-\alpha)\omega^{\frac{\alpha}{1-\alpha}})\exp(-\lambda r)$. Since the exact optimal solution of $\inf_{\lambda,\omega} F(\lambda, \omega)$ is not a closed-form expression, one can seek for the suboptimal values instead.

To this end, let 
\begin{align}\label{eq:param}
    \hat{\lambda} = \frac{k\,r^{\frac{1-\alpha}{1+\alpha}}}{L_\alpha^2 d^\alpha \eta^{\alpha+1}}, ~~\hat{\omega} = c\,r^\frac{2(1-\alpha)}{1+\alpha}\eta^{{1-\alpha}}d^{{1-\alpha}} ~~~\text{ for some } k>0, ~~c>0 
\end{align}

It follows that 
\begin{align}
F({\hat{\lambda},\hat{\omega}}) & = (1-\frac{\pi^2\alpha k^2}{4cL_\alpha^2d^{\alpha+1}\eta^{\alpha+1}})^{-d/2}\exp\left(-(k-\frac{\pi^2}{8}k^2(1-\alpha)c^{\frac{\alpha}{1-\alpha}}) \frac{r^{\frac{2}{\alpha+1}}}{L_\alpha^2d^\alpha\eta^{\alpha+1}}\right).
\end{align}
Similarly, let $\frac{\pi^2\alpha k^2}{4cL_\alpha^2d^{\alpha}\eta^{\alpha+1}} = \epsilon  \in (0,d)$, and plugging $k = \sqrt{\frac{4c\epsilon\eta^{\alpha+1}d^\alpha L_\alpha^2}{\pi^2\alpha}}$ into $F({\hat{\lambda},\hat{\omega}})$ yields
\begin{align}
F({\hat{\lambda},\hat{\omega}}) = (1-\frac{\epsilon}{d})^{-d/2}\exp\left(-D\frac{r^{\frac{2}{\alpha+1}}}{L_\alpha^2d^\alpha\eta^{\alpha+1}}\right) 
\end{align}
with 
\begin{equation}\nonumber
    D = \sqrt{\frac{4\epsilon L_\alpha^2d^\alpha \eta^{\alpha+1}}{\pi^2\alpha}}c^{1/2}-\frac{(1-\alpha)\epsilon L_\alpha^2 d^\alpha \eta^{\alpha+1}}{2\alpha}c^{\frac{1}{1-\alpha}}.
\end{equation}
Directly optimizing $D$ as a function of $c$ gives 
$$
\max D = (1+\alpha)\left(\frac{1}{\alpha}\right)^{\frac{\alpha}{1+\alpha}}\left(\frac{1}{\pi^2}\right)^{\frac{1}{1+\alpha}}2^{\frac{1-\alpha}{1+\alpha}}\epsilon^{\frac{\alpha}{\alpha+1}}{L_\alpha^{\frac{2\alpha}{\alpha+1}}d^{\frac{\alpha^2}{\alpha+1}}\eta^\alpha}.
$$
This implies
\begin{equation}\label{eq:theorem4_1}
    \Pr(f(X)-\mE(f(X))\geq r) \leq (1-\frac{\epsilon}{d})^{-d/2}\exp\left(-\frac{C\epsilon^{\frac{\alpha}{\alpha+1}}r^{\frac{2}{\alpha+1}}}{L_\alpha^{\frac{2}{\alpha+1}}d^{\frac{\alpha}{\alpha+1}}\eta}\right), \: \forall \epsilon \in (0,d),
\end{equation}
with $C$ being $(1+\alpha)\left(\frac{1}{\alpha}\right)^{\frac{\alpha}{1+\alpha}}\left(\frac{1}{\pi^2}\right)^{\frac{1}{1+\alpha}}2^{\frac{1-\alpha}{1+\alpha}}$.

It is noteworthy that \eqref{eq:theorem4_1} can recover the results for both smooth functions and Lipschitz functions. More precisely, plugging $\alpha=1$ into \eqref{eq:theorem4_1} leads to \eqref{eq:alpha1_conc}. Moreover, if $\alpha = 0$, by choosing $\epsilon \rightarrow 0$, \eqref{eq:theorem4_1} coincides with \eqref{eq:alpha0_conc}.

\end{enumerate}

\end{proof}

\subsection{Proof of Corollary \ref{cor:composite_concentration}}

\begin{proof}
By Theorem \ref{thm:concentration} and a probabilistic uniform bound, we have for any $\sum_{j=1}^n w_j =1, ~w_j \ge 0$
\begin{align}
\Pr( \ell(X)-\mE(\ell(X))\geq r) 
\leq \left(1-\frac{\epsilon}{d} \right)^{-\frac{d}{2}}
\sum_{j=1}^n 
 \exp\left(-\frac{C_j \epsilon^{\frac{{\alpha_j}}{1+{\alpha_j}}}(w_j r)^{\frac{2}{1+{\alpha_j}}}}{L_{\alpha_j}^{\frac{2}{1+{\alpha_j}}}d^{\frac{{\alpha_j}}{1+{\alpha_j}}}\eta}\right), ~~ \forall \epsilon \in (0,d).
\end{align}
\end{proof}

\subsection{Relations to existing concentration inequalities}\label{sec:conc_relation}
We now compare Theorem \ref{thm:concentration} with other existing concentration inequalities. In this section, we assume $X \sim \cN( 0, \eta \ide)$.
\paragraph{Lipschitz function $\alpha=0$.} 
Taking $\epsilon \rightarrow 0 $,
our concentration can be simplified to 
\begin{align}\label{eq:conc_alpha0}
\Pr( \ell(X)-\mE(\ell(X))\geq r) 
\leq 
\exp\left(-\frac{2
 r^{2}}{\pi^2 L_0^{2} 
\eta}\right).
\end{align}
Obviously, it has a sub-optimal coefficient because based on the entropy argument~\citep{boucheron2013concentration}, one can obtain
\begin{align}
\Pr( \ell(X)-\mE(\ell(X))\geq r) 
\leq 
\exp\left(-\frac{
r^{2}}{2 L_0^{2} 
\eta}\right).    
\end{align}
\paragraph{Smooth function $\alpha=1$.}
Taking $\epsilon =0.5 $,
our concentration bound is simplified to 
\begin{align}\label{eq:smooth_conc_ours}
\Pr( \ell(X)-\mE(\ell(X))\geq r) 
\leq 1.5
\exp\left(-\frac{
0.45 r}{ L_1 \sqrt{d}
\eta}\right).
\end{align}
Suppose $Y = \sum_{i=1}^d Y_i^2, ~Y_i \sim \cN( 0, \ide)$, then $Y \sim \chi^2(d)$ follows the chi-squared distribution, and $\mE [Y] =d $. Thus, taking  $\ell(X) = \|X\|^2 $,
the \textbf{Laurent-Massart bound}~\citep{laurent2000adaptive} can be equivalently written as
\begin{align} 
    \Pr( \|X\|^2 - \mE(\|X\|^2) \ge 2 \sqrt{\mE(\ell(X)) r } + 2r  ) \le \exp(-r).
\end{align}
This further implies, if $X \sim \cN(0, \eta \ide)$, it holds that
\begin{align}\label{eq:smooth_conc_chi}
    \Pr( \|X\|^2 - \mE(\|X\|^2) \ge r )
    \le \exp\left(- \frac{ d+r /\eta - \sqrt{d(d+2r /\eta )} }{2} \right),
\end{align}
Take $L_1 = 2, \eta =1$, we compare \eqref{eq:smooth_conc_ours} and \eqref{eq:smooth_conc_chi}  in terms of $d$. Since $(d+r - \sqrt{d(d+2r)}) \sqrt{d}/ r \rightarrow 0$ for a fixed $r$,
 our bound is even tighter on dimension when $d$ is large enough. 

On the other hand, we compare with \textbf{Hanson-Wright inequality}~\citep{rudelson2013hanson}: Let $A$ be an $d \times d $ matrix, 
\begin{align}\label{eq:hw}
\Pr( X^\top A X - \mE  (X^\top A X)  > r  )\le \exp \left( -c \min\left( \frac{4 r^2}{\eta^2 \|A\|^2_{\text{HS}} }  ,  \frac{2 r}{ \eta \|A\| } \right) \right).
\end{align}
Take $\ell(X) =  X^\top A X$, and
suppose $A$ is positive semi-definite, then $L_1 = \lambda_{max} (2A) = 2 \|A\| $, thus
our bound becomes 
\begin{align}
\Pr(X^\top A X -\mE( X^\top A X )\geq r) 
\leq 1.5
\exp\left(-\frac{
0.225 r}{ \|A\| \sqrt{d} \eta }\right).  
\end{align}
When $r$ is large enough, $2r/\eta$ would dominate, and Hanson-Wright is tighter than our bound in terms of $d$. When $r$ is in 
  a small range,   $r^2/d \eta^2 $ dominates, and our bound is tighter on dimension.

\subsection{Low range concentration inequality has sub-Gaussian tail}
In this part, we show that no matter what $0 \le \alpha \le 1$ is, we can always get a sub-Gaussian concentration when $r$ is in a low range.
\begin{proposition}\label{prop:low_range}
Let $X \sim \cN(m,\eta\ide)$ be a Gaussian random variable in $\Rd$, and let $\ell$ be an $L_\alpha$-$\alpha$-semi-smooth function. Assume $ \ell'(m) = 0$. Then for any $ ~0 \le \alpha \le 1$, 
\begin{align}\label{eq:low_range}
0 < r \le \frac{\pi L_\alpha  d^{\frac{1+\alpha}{2}} \eta^{\frac{1+\alpha}{2}} }{\sqrt{\alpha 2^{\alpha}}},    
\end{align}
 we have 
\begin{align}\label{eq:low_range_conc}
\Pr( \ell(X)-\mE(\ell(X))\geq r) 
\leq 
\exp\left(-\frac{ 
r^{2}}{\pi^2 L_\alpha^{2 }d^{{\alpha}}\eta^{1+\alpha}}\right).
\end{align}
\end{proposition}
\begin{proof}
It suffices to prove the case $\eta =1$ because we can define $\hat \ell(X)= \ell(\sqrt{\eta} X) $ and correspondingly $\hat L_\alpha =L_\alpha \eta^{{(1+\alpha)}/2} $. Similar to the proof of Theorem \ref{thm:concentration}, we use  Maurey and Pisier argument and apply Young's inequality, and obtain
\begin{align}\label{eq:conc_template}
&\Pr(\ell(X)-\mE(\ell(X))\geq r) \le \inf_{\lambda > 0,~ \omega >0} F(\lambda, \omega) \\
\text{where} ~~
F(\lambda, \omega ): = & \left(1-\frac{\pi^2}{4}L_\alpha^2
\lambda^2\frac{\alpha}{\omega} \right)^{-\frac{d}{2}}
\exp\left(\frac{\pi^2}{8}L_\alpha^2
\lambda^2(1-\alpha)\omega^{\frac{\alpha}{1-\alpha}}\right) \exp(-\lambda r).    
\end{align}
We define $H = \frac{\pi^2}{8} \lambda^2 L_\alpha^2 $. Firstly, we assume that (will prove in the end) 
\begin{align} \label{eq:half_constraint}
\frac{2H \alpha}{\hat \omega } \in [0,0.5],
\end{align} 
which suffices to conclude
\begin{align}
\left(1-\frac{\pi^2}{4}L_\alpha^2
\lambda^2\frac{\alpha}{\hat \omega} \right)^{-\frac{d}{2}} 
= \left(1-  \frac{2H \alpha}{\hat \omega} \right)^{-\frac{d}{2}} \le \exp\left(\frac{2 H \alpha d}{\hat \omega}\right) .
\end{align}
Plugging this inequality, and the value $\hat \omega = (2d)^{1-\alpha} $ into $F(\lambda , \omega)$, we obtain
\begin{align}
 F(\lambda , \hat \omega ) \le &\exp\left(\frac{2 H \alpha d}{(2d)^{1-\alpha} } + H (1-\alpha) (2d)^\alpha - \lambda r \right) \\
\le & \exp(2H d^\alpha - \lambda r) =  \exp \left(\frac{\pi^2}{4} \lambda^2 L_\alpha^2 d^\alpha - \lambda r \right).
\end{align}
We then take $\hat \lambda = \frac{2 r}{ \pi^2 L_\alpha^2 d^\alpha } $, and obtain $F( \hat \lambda, \hat \omega ) = \exp\left(-\frac{ r^{2}}{\pi^2 L_\alpha^{2 }d^{{\alpha}}}\right).$
This together with \eqref{eq:conc_template} gives \eqref{eq:low_range_conc}.
Finally, to make \eqref{eq:half_constraint} hold, we need to insure 
\begin{align}
    \frac{2H \alpha}{ \hat \omega} = \frac{ \pi^2 \lambda^2 L_\alpha^2 \alpha }{4 \hat \omega } = \frac{r^2 \alpha }{ \pi^2 L_\alpha^2 2^{1-\alpha} d^{1+\alpha} } \le 0.5,
\end{align}
which is equivalent to the constraint \eqref{eq:low_range}.
\end{proof}

Proposition \ref{prop:low_range} can recover the Hanson-Wright inequality \eqref{eq:hw} in the low range. 
If we take $A = \ide$, then \eqref{eq:hw} becomes
\begin{align}
  \Pr( \|X\|^2 - \mE( \|X\|^2 ) \ge r ) \le \exp \left( - c \min\left( \frac{4r^2}{  d \eta^2 }, \frac{2r}{  \eta} \right) \right).
\end{align}
This implies that, when $r \le \Theta(d \eta )$, $$\Pr( \|X\|^2 - \mE( \|X\|^2 ) \ge r ) \le \exp \left( -   \frac{cr^2}{  d \eta^2 }\right) . $$
It is easy to check that Proposition \ref{prop:low_range} shows exactly the same result when $\alpha=1$.

\section{Proof of main results}\label{sec:main_result_proof}
\subsection{Proof of Theorem \ref{thm:rgo} }

\begin{proof}
Recall that the return of Algorithm \ref{alg:rgo} follows the distribution $ \hat \pi^{X|Y}$, and the target distribution of RGO is $\pi^{X|Y} \propto \exp(-g(x) - \frac{1}{2 \eta} \|x - x_y\|^2 ) $.
According to Lemma \ref{lem:dist_relation}, we have 
\begin{align}
 \| \pi^{X|Y} - \hat \pi^{X|Y} \|_\TV 
 & =  \mE \left| \frac{V}{ \mE V } - \frac{\overline V}{ \mE \overline V}  \right| 
 \le  \mE \left| \frac{V}{ \mE V } - \frac{\overline V}{ \mE  V}  \right| + \mE \left| \frac{\overline V}{ \mE  V } - \frac{\overline  V}{ \mE \overline V}  \right|  \\
 & \le \frac{\mE |V - \overline V|}{ |\mE V|} +  \frac{ \mE  [\overline V] |\mE [V - \overline V]|}{ |\mE V| | \mE  \overline V| }
 \le \frac{2\mE |V - \overline V|}{ |\mE V|}.
\end{align}
Then by Cauchy-Schwartz inequality on the Hilbert space $L^2(\phi)$, we have
\begin{align}
    & |\mE V| = \mE V = \mE_{x \sim \phi} \exp(-g(x)) \mE_{x \sim \phi} \exp(g(x)) = \|\exp(-g(x)/2)\|_{L^2(\phi)}^2 \|\exp(g(x)/2)\|_{L^2(\phi)}^2  \nonumber \\
    \ge & \< \exp(-g(x)/2) , \exp(g(x) /2)\>^2 = [\mE_{x\sim \phi}(\exp(-g(x)/2) \exp(g(x)/2))]^2 =1. \label{eq:bound_V}
\end{align}
Since $\rho = \exp( g(z) -g(x) )$ is always non-negative, there is
    \begin{align}
       \mE|V- \overline V | =  \mE| \mE [\rho | x] - \mE [\bar \rho | x] | =  \mE[(\rho -2) \one_{\rho \ge 2} ]
       \le 
       \mE[\rho \one_{\rho \ge 2} ].
    \end{align}
Denote $\Delta = g(z) - g(x)$ and $\bar \Delta = \Delta / \log 2$, we have
\begin{align}
\mE[\rho \one_{\rho \ge 2} ] 
& =
\mE[\exp( \Delta) \one_{\rho \ge 2} ] = \mE[\exp( \Delta) \one_{\Delta \ge \log 2 } ]
= \mE[\exp( \bar \Delta \log 2) \one_{ \bar \Delta \ge 1 } ]  \\
& \le \sum_{i=1}^\infty \exp( (i+1)\log 2) \Pr(\bar \Delta \ge i) .
\end{align}
Note that $g(x) = f(x) - \<  f'(x_y), x\>$ satisfies that $ g'(x_y) =0 $, and $\mE[x] = x_y$. Moreover, $g(x)$ is also $L_\alpha$-$\alpha$-semi-smooth because $ g'(x_1) -  g'(x_2) =  f'(x_1) - f'(x_2)  $ for any $x_1,x_2$.
We also have the inequality $(1-0.5/a)^{-a/2} \le 1.5$  when $a \ge 1$.
We plug $\epsilon = 0.5 $ into  Theorem \ref{thm:concentration}, and obtain that, for $\forall r >0$,
\begin{align}\label{eq:const_concentrate}
\Pr[g(x) - \mE g(x) \ge r ] \le \frac{3}{2} \exp \left(  - \frac{
\Cepshalf
r^{\frac{2}{1+\alpha}}
 }{
 L_\alpha^{\frac{2}{\alpha+1}}d^{\frac{\alpha}{\alpha+1}}\eta
 } 
 \right) .
\end{align}
This further implies that 
\begin{align}
\Pr[g(z) - g(x) \ge r ] 
& = \Pr[g(z) - \mE g(z) + \mE g(x) -  g(x) \ge r ] \\
& \le \Pr \left[g(z) - \mE g(z) \ge \frac{r}{2} \right]+ \Pr \left[\mE g(x) -  g(x) \ge \frac{r}{2} \right] \\
& = \Pr \left[g(z) - \mE g(z) \ge \frac{r}{2} \right]+ \Pr \left[ -  g(x) - \mE [-g(x)] \ge \frac{r}{2} \right] \\
& \le 3 \exp \left( - \frac{
\Cepshalf
\left(\frac{r}{2}\right)^{\frac{2}{1+\alpha}}
 }{
 L_\alpha^{\frac{2}{\alpha+1}}d^{\frac{\alpha}{\alpha+1}} 
 \eta } 
 \right) .
\end{align}
Thus 
$
\Pr[\bar \Delta \ge i ] \le 
3 \exp \left( -\frac{
\Cepshalf
\left(\frac{i \log 2}{2}\right)^{\frac{2}{1+\alpha}}
 }{
 L_\alpha^{\frac{2}{\alpha+1}}d^{\frac{\alpha}{\alpha+1}} 
 \eta } 
 \right) . $
Denote $$C_\eta : = \frac{
\Ceta
}{
  L_\alpha^{\frac{2}{\alpha+1}}d^{\frac{\alpha}{\alpha+1}}
 \eta } -1 . $$ 
 Since $\log 2 < 1$ and  $i^{\frac{2}{1+\alpha}} \ge i$ for any $i \ge 1$, $0 \le \alpha \le 1$, we have that 
 \begin{align}
& \sum_{i=1}^\infty \exp( (i+1)\log 2) \Pr(\bar \Delta \ge i) 
\le  6 \sum_{i=1}^\infty 
\exp \left( i - \frac{
\Cepshalf
\left(\frac{i \log 2}{2}\right)^{\frac{2}{1+\alpha}}
 }{
 L_\alpha^{\frac{2}{\alpha+1}}d^{\frac{\alpha}{\alpha+1}}
 \eta } 
 \right) \\
 \le  & 6 \sum_{i=1}^\infty 
\exp \left(  - C_\eta
 i \right)  = \frac{6}{\exp(C_\eta) -1 } \le \frac{\zeta }{2} \label{eq:bound_E_V_bar_V}
 \end{align}
 by the choice 
 \begin{align} \label{eq:eta_full}
     \eta \le  \frac{
\Ceta
}{
  L_\alpha^{\frac{2}{\alpha+1}}d^{\frac{\alpha}{\alpha+1}} (1+ \log (1+ 12/\zeta )) }. 
 \end{align}
 Since 
\begin{align}\label{eq:alpha_bound}
  0.024 < 
\Ceta
< 0.16   
 \end{align}
when $0 \le \alpha \le 1 $, 
the bound
$$\eta \le \frac{1}{\Consteta L_\alpha^{\frac{2}{\alpha+1}}d^{\frac{\alpha}{\alpha+1}}(1+ \log (1+ 12/\zeta )) } $$ suffices.
When $\alpha=0$, it can recover the bound in \citet{gopi2022private}. 

Finally, by Lemma \ref{lem:dist_relation},
the acceptance probability is $\frac{1}{2} \mE[\overline V] $. Since $\zeta< 1 $, we have 
\begin{align}\label{eq:low_bound_bar_V}
 \mE[\overline V] \ge \mE[ V] -  \mE|V - \overline V| \ge  1- \frac{\zeta}{2} \ge \frac{1}{2} .   
\end{align}
Thus, the expected number of the iterations for the rejection sampling step is $\cO(1)$.
\end{proof}

\subsection{Proof of Theorem \ref{thm:rgo_w2} }

\begin{proof}
It suffices to prove the claim:
\textit{If the step size 
\begin{align}\label{eq:w2_eta_simple}
\eta \le \min \left(  \frac{1}{
 \Consteta L_\alpha^{\frac{2}{\alpha+1}}d^{\frac{\alpha}{\alpha+1}}
(2+ \log( 1+ \frac{192  (d^2 + 2d)}{ \zeta^2} )) } , 1\right)  ,  
\end{align}
then Algorithm \ref{alg:rgo} returns a random point $x$ that has $\zeta$ \textbf{squared}  Wasserstein-2 distance to the distribution $\pi^{X|Y}$.  Furthermore, if ~$0< \zeta <8d $, then the algorithm access $\cO(1)$ many $f(x)$ in expectation.}

Recall that the return of Algorithm \ref{alg:rgo} follows the distribution $ \hat \pi^{X|Y}$, and the target distribution of RGO is $\pi^{X|Y} \propto \exp(-g(x) - \frac{1}{2 \eta} \|x - x_y\|^2 ) $.
By Lemma \ref{lem:tv_w2} (\citet[Proposition 7.10]{villani2021topics}) and triangular inequality,
    \begin{align}
W_2^2( \pi^{X|Y} , \hat \pi^{X|Y} ) & \le 2 \left\| \| \cdot - x_y \|_2^2 (\pi^{X|Y} - \hat \pi^{X|Y} ) \right\|_\TV 
= 2 \mE \left( \|x - x_y \|_2^2 \left| \frac{V}{\mE V} - \frac{\overline V}{\mE \overline V} \right| \right) \\
& \le 2  \mE \|x - x_y \|_2^2 \left( \left| \frac{V}{ \mE V } - \frac{\overline V}{ \mE  V}  \right| +  \left| \frac{\overline V}{ \mE  V } - \frac{\overline  V}{ \mE \overline V}  \right| \right)  \\
&\le \frac{ 2 \mE [ \|x - x_y\|^2 |V - \overline V| ] }{ \mE V } 
+  \frac{ 2 \mE [\|x- x_y\|^2 \overline V |\mE [V - \overline V]|]}{ \mE V  \mE  \overline V } \\
& \overset{\eqref{eq:bound_V}}{\le} 2 \mE [ \|x - x_y\|^2 |V - \overline V| ] 
+  \frac{ 2 \mE |V - \overline V| \cdot \mE [\|x- x_y\|^2 \overline V ]}{  \mE  \overline V } . \label{eq:w2_2_terms}
    \end{align}
Firstly, by 
Cauchy-Schwartz inequality,
\begin{align}\label{eq:w2_1st_term_cs}
 \mE [ \|x - x_y\|^2 |V - \overline V| ] \le \left( \mE  \|x - x_y\|^4 \mE |V - \overline V|^2  \right)^\frac{1}{2} .
\end{align}
 Denote $G \sim \cN(0, \ide )$.  Since $x \sim \cN(x_y, \eta \ide )$ and the constraint $\eta \le 1$, we have 
\begin{align}
\mE  \|x - x_y\|^4 
& =  \eta^2 \mE   \|G\|^4  =  \eta^2 \left( \sum_i \mE [ G_i^4] + \sum_{i \neq j} [\mE  ~G_i^2~] [\mE  ~G_j^2] \right) \\
& = \eta^2(3d + d^2-d) =  \eta^2(d^2 + 2d) \le  d^2 + 2d. \label{eq:4_moment}
\end{align}
On the other hand, following the same logic of bounding $\mE [V - \overline V]$ in Theorem \ref{thm:rgo}, we have
\begin{align}
\mE |V - \overline V|^2 
& =    \mE | \mE[ \rho - \bar \rho  |x ]  |^2
\le   \mE \mE[ | \rho - \bar \rho |^2 |x ]  
\le  \mE | \rho - \bar \rho |^2   
\le \mE [\rho^2 \one_{\rho \ge 2}  ] \\
& \le \sum_{i=1}^\infty \exp( 2 (i+1)\log 2) \Pr(\bar \Delta \ge i) . \label{eq:v_diff_squ}
\end{align}
Denote 
$$C_\eta : = \frac{
\Ceta
}{
  L_\alpha^{\frac{2}{\alpha+1}}d^{\frac{\alpha}{\alpha+1}}
 \eta } - 2.  $$ 
 Since the bound \eqref{eq:alpha_bound},
we can verify that the assumption \eqref{eq:w2_eta_simple}
satisfies 
\begin{align}\label{eq:w2_eta_choice}
\eta \le \frac{
\Ceta
}{
  L_\alpha^{\frac{2}{\alpha+1}}d^{\frac{\alpha}{\alpha+1}}
(2+ \log( 1+ \frac{192  (d^2 + 2d)}{ \zeta^2} )) }, 
 \end{align}
By the concentration inequality \eqref{eq:const_concentrate}, following the same proof of bounding \eqref{eq:bound_E_V_bar_V},
we have
 \begin{align}
& \sum_{i=1}^\infty \exp( 2 (i+1)\log 2) \Pr(\bar \Delta \ge i) \\
\overset{\eqref{eq:const_concentrate}}{\le}   & 12 \sum_{i=1}^\infty 
\exp \left( 2 i - \frac{
\Cepshalf
\left(\frac{i \log 2}{2}\right)^{\frac{2}{1+\alpha}}
 }{
 L_\alpha^{\frac{2}{\alpha+1}}d^{\frac{\alpha}{\alpha+1}}
 \eta } 
 \right) \\
 \le  & 12 \sum_{i=1}^\infty 
\exp \left(  - C_\eta
 i \right)  = \frac{12}{\exp(C_\eta) -1 } 
 \overset{\eqref{eq:w2_eta_choice}}{\le}  \frac{\zeta^2}{ 
 16 (d^2 + 2d) }  .
 \end{align}
 This, together with the inequalities \eqref{eq:w2_1st_term_cs}~\eqref{eq:4_moment}, gives that 
 \begin{align}\label{eq:w2_term1}
  \mE [ \|x - x_y\|^2 |V - \overline V| ]  \le \frac{\zeta}{4} .   
 \end{align}
 Now we bound the second term in \eqref{eq:w2_2_terms}. By the assumption 
 \eqref{eq:w2_eta_simple},  we have
\begin{align}\label{eq:w2_eta_choice2}
      \eta \le \frac{1}{\Consteta L_\alpha^{\frac{2}{\alpha+1}}d^{\frac{\alpha}{\alpha+1}}(1+ \log (1+ 
 96 d / \zeta
 )) } .
 \end{align}
 Thus, following the same steps in bounding
 \eqref{eq:bound_E_V_bar_V} and \eqref{eq:low_bound_bar_V}, we obtain 
\begin{align}\label{eq:V_and_difference}
    \mE |V - \overline V| \le
    \frac{\zeta}{ 16 d}  ~~\text{and}~~
    \mE[ \overline V] \ge \frac{1}{2}.
\end{align}
Next, since $0 \le \overline{V} := \mE[\bar \rho | x] \le 2$, there is
\begin{align}\label{eq:weighted_bar_V}
\mE [\|x- x_y\|^2 \overline V ] 
\le 2 \mE [\|x- x_y\|^2 ]  
= 2 \eta d.
\end{align}
By the choice $\eta \le 
1
$ 
  and the bounds \eqref{eq:V_and_difference} and \eqref{eq:weighted_bar_V}, we have 
\begin{align}
 \frac{ \mE |V - \overline V| \cdot \mE [\|x- x_y\|^2 \overline V ]}{  \mE  \overline V } 
 \le \frac{\zeta}{4} .
\end{align}
This, together with \eqref{eq:w2_2_terms} and \eqref{eq:w2_term1}, gives that $W_2^2( \pi^{X|Y} , \hat \pi^{X|Y} ) \le \zeta .$

Finally, because of the choice $\zeta< 8d $, and the bound \eqref{eq:V_and_difference}, we have  $  \mE |V - \overline V| 
\le
    \frac{1}{ 2}  $.
Thus, the acceptance probability $\frac{1}{2} \mE[\overline V]
 \ge \frac{1}{2}( \mE[ V] -  \mE|V - \overline V| )
\ge \frac{1}{4} $.  So the expectation of iterations in rejection sampling follows as $\cO(1)$.
\end{proof}

\subsection{Proof of Proposition \ref{prop:strongly_convex}}

\begin{proof}
{\bf (For TV distance)}
According to Theorem \ref{thm:rgo}, if choosing $\eta \le \frac{1}{\Consteta
L_1 \sqrt{d}
(1+ \log (1+ 24 T /\delta )) } $, we can guarantee $ \left\| \hat \pi^{X|Y} (\cdot |y) -  \pi^{X|Y} (\cdot |y)  \right\|_\TV \le  \delta /(2T) $ for $\forall ~y$. Then by Lemma \ref{lem:tv_telecope}, we have $ \| \hat \mu_T - \mu_T \|_\TV \le \frac{\delta}{2}.$

By Pinsker's inequality, there is $\| \mu_T - \nu\|_\TV \le \sqrt{2 H_\nu(\mu_T) }$.
Then by Theorem \ref{thm:proximal_convergence} \citep[Theorem 3]{chen2022improved},  Algorithm \ref{algo:AlternatingSampler} returns a random point $x_T$ that satisfies
\begin{align}
H_\nu(\mu_T) \le \frac{H_\nu(\mu_0)}{(1+ \eta\beta )^{2T}} \le \frac{\delta^2}{8}
\end{align}
in 
$T \ge  {\log \left( \frac{2}{ \delta } \sqrt{ {2 H_\nu (\mu_0)} } \right)}/{\log(1+\beta \eta)} $
steps. Thus $\|\mu_T - \nu\| \le \delta /2$.

Putting two together, we have
\begin{align}
\| \hat \mu_T - \nu \|_\TV  \le   \| \hat \mu_T - \mu_T \|_\TV  + \| \mu_T - \nu \|_\TV  \le \delta .
\end{align}
Since $ \eta \beta = \cO(1)$, we have $\log(1+ \eta \beta ) = \cO( \eta \beta )$.
Thus, plugging in the value of $\eta$, we need 
$$T =   \cO  \left( \frac{ L_1 \sqrt{d}}{ \beta }  \log \left( \frac{L_1 \sqrt{d}}{ \beta \delta } \right)  \log \left( \frac{ \sqrt{H_\nu (\mu_0)}}{ \delta} \right) \right) $$ steps. Each step accesses only $\cO(1)$ many $f(x)$ in expectation because of Theorem \ref{thm:rgo}.  

{\bf (For Wasserstein distance)} Since $\nu$ is $\beta$-strongly-log-concave, it satisfies Talagrand inequality
\begin{align}
      W_2(\mu_T, \nu) \le \sqrt{\frac{2}{\beta} H_\nu(\mu_T)}.
\end{align}
      Next, Theorem \ref{thm:proximal_convergence} \citep[Theorem 3]{chen2022improved} guarantees that if 
    $T \ge  {\log \left( \frac{2}{ \delta } \sqrt{ \frac{2 H_\nu (\mu_0)}{ \beta} } \right)}/{\log(1+\beta\eta)} $
    then
\begin{equation}
H_\nu (\mu_T) \leq \frac{ H_\nu \left({\mu}_0 \right)}{(1+ \beta \eta)^{2T}} 
\le \frac{ \beta \delta^2}{8} .
\end{equation}    
Thus, $ W_2(\mu_T, \nu)\le \frac{\delta}{2} . $
On the other hand, according to Theorem \ref{thm:rgo_w2}, if choosing $$
\eta \le \min \left(  \frac{1}{\Consteta L_1 d^{\frac{1}{2}}(2+ \log( 1+ {3100  (d^2 + 2d)T^4}/{ \delta^4} )) } , 1\right)
,
$$ we can guarantee 
$ W_2(\hat \pi^{X|Y} (\cdot |y) ,  \pi^{X|Y} (\cdot |y) ) \le  \delta /2T $ for $\forall ~y$. Then Lemma \ref{lem:w2_telecope} guarantees that $W_2( \mu_T , \hat \mu_T  ) \le \frac{\delta}{2} .$ 
Putting two pieces together, we get 
\begin{align}
    W_2( \hat \mu_T, \nu ) \le W_2( \hat \mu_T, \mu_T ) + W_2( \mu_T, \nu ) \le \delta.
\end{align}
Since $\beta \eta = \cO(1)$, we have $\log(1+ \beta \eta) = \cO(\beta \eta)$.
So, plugging in the value of $\eta$, we only need  $T =\cO \left( \frac{1}{\beta \eta} {\log \left( \frac{2}{ \delta } \sqrt{ \frac{2 H_\nu (\mu_0)}{ \beta} } \right)} \right) 
= \cO \left( \frac{ L_1 \sqrt{d}}{\beta} \log\left( \frac{L_1 d }{\beta \delta} \right) {\log \left( \frac{1}{ \delta } \sqrt{ \frac{ H_\nu (\mu_0)}{ \beta} } \right)} \right) 
$ steps.
Each step accesses only $\cO(1)$ many $f(x)$ in expectation because of Theorem \ref{thm:rgo_w2}.
\end{proof}

\subsection{Proof of Proposition \ref{prop:convex}}

\begin{proof}
According to Theorem \ref{thm:rgo}, if choosing $\eta \le
\frac{1}{\Consteta L_\alpha^{\frac{2}{\alpha+1}}d^{\frac{\alpha}{\alpha+1}}(1+ \log (1+ 24 T/\delta )) } $, we can guarantee $ \left\| \hat \pi^{X|Y} (\cdot |y) -  \pi^{X|Y} (\cdot |y)  \right\|_\TV \le  \delta /(2T) $ for $\forall ~y$. Then by Lemma \ref{lem:tv_telecope}, we have $ \| \hat \mu_T - \mu_T \|_\TV \le \frac{\delta}{2}.$

By Pinsker's inequality, there is $\| \mu_T - \nu\|_\TV \le \sqrt{2 H_\nu(\mu_T) }$.
Since $f(x)$ is convex and $L_1$-smooth,
by Theorem \ref{thm:proximal_convergence} \citep[Theorem 2]{chen2022improved},  Algorithm \ref{algo:AlternatingSampler} returns a random point $x_T$ that satisfies
\begin{align}
H_\nu(\mu_T) \le \frac{W_2^2(\mu_0, \nu)}{T \eta} \le \frac{\delta^2}{8}
\end{align}
in 
$T = \cO \left( {W_2^2(\mu_0, \nu)}/{(\delta^2 \eta)} \right) $  steps. Thus $\|\mu_T - \nu\| \le \delta /2$.

Putting two together, we have
\begin{align}
\| \hat \mu_T - \nu \|_\TV  \le   \| \hat \mu_T - \mu_T \|_\TV  + \| \mu_T - \nu \|_\TV  \le \delta .
\end{align}
Thus, plugging in the value of $\eta$, we need 
$$
T= \cO \left( \frac{W_2^2(\mu_0, \nu) L_\alpha^{ \frac{2}{\alpha+1} } d^{\frac{\alpha}{\alpha+1}}
}{\delta^2 } 
\log \left( \frac{W_2^2(\mu_0, \nu) L_\alpha^{ \frac{2}{\alpha+1} } d^{\frac{\alpha}{\alpha+1}}}{\delta^3 }\right)
\right)
$$ steps. Each step accesses only $\cO(1)$ many $f(x)$ in expectation because of Theorem \ref{thm:rgo}. 
\end{proof}

\subsection{Proof of Proposition \ref{prop:lsi}}

\begin{proof}
According to Theorem \ref{thm:rgo}, if choosing $\eta \le
\frac{1}{\Consteta L_1 \sqrt{d}(1+ \log (1+ 24 T/\delta )) } $, we can guarantee $ \left\| \hat \pi^{X|Y} (\cdot |y) -  \pi^{X|Y} (\cdot |y)  \right\|_\TV \le  \delta /(2T) $ for $\forall ~y$. Then by Lemma \ref{lem:tv_telecope}, we have $ \| \hat \mu_T - \mu_T \|_\TV \le \frac{\delta}{2}.$

By Pinsker's inequality, there is $\| \mu_T - \nu\|_\TV \le \sqrt{2 H_\nu(\mu_T) }$.
Then by Theorem \ref{thm:proximal_convergence} \citep[Theorem 3]{chen2022improved},  Algorithm \ref{algo:AlternatingSampler} returns a random point $x_T$ that satisfies
\begin{align}
H_\nu(\mu_T) \le \frac{H_\nu(\mu_0)}{(1+ \eta C_\LSI )^{2T}} \le \frac{\delta^2}{8}
\end{align}
in 
$T \ge  {\log \left( \frac{2}{ \delta } \sqrt{ {2 H_\nu (\mu_0)} } \right)}/{\log(1+ C_\LSI\, \eta)} $
steps. Thus $\|\mu_T - \nu\| \le \delta /2$.

Putting two together, we have
\begin{align}
\| \hat \mu_T - \nu \|_\TV  \le   \| \hat \mu_T - \mu_T \|_\TV  + \| \mu_T - \nu \|_\TV  \le \delta .
\end{align}
Since $ \eta C_\LSI = \cO(1)$, we have $\log(1+ \eta C_\LSI ) = \cO( \eta C_\LSI )$.
Thus, plugging in the value of $\eta$, we need 
$$T =   \cO  \left( \frac{ L_1 \sqrt{d}}{ C_\LSI } \log \left( \frac{L_1 \sqrt{d} }{ C_\LSI \delta} \right)  \log \left( \frac{ \sqrt{H_\nu (\mu_0)}}{ \delta} \right) \right) $$ steps. Each step accesses only $\cO(1)$ many $f(x)$ in expectation because of Theorem \ref{thm:rgo}. 
\end{proof}

\subsection{Proof of Proposition \ref{prop:pi}}

\begin{proof}
According to Theorem \ref{thm:rgo}, if choosing $\eta \le
\frac{1}{\Consteta L_\alpha^{\frac{2}{\alpha+1}}d^{\frac{\alpha}{\alpha+1}}(1+ \log (1+ 24 T/\delta )) } $, we can guarantee $ \left\| \hat \pi^{X|Y} (\cdot |y) -  \pi^{X|Y} (\cdot |y)  \right\|_\TV \le  \delta /(2T) $ for $\forall ~y$. Then by Lemma \ref{lem:tv_telecope}, we have $ \| \hat \mu_T - \mu_T \|_\TV \le \frac{\delta}{2}.$

Together with Pinsker's inequality, \citet{
nishiyama2020relations} implies that $\| \mu_T - \nu\|_\TV \le \sqrt{2 \log( (1+ \chi_\nu^2 (\mu_T) )) }$.
Then by \citet[Theorem 4]{chen2022improved},  Algorithm \ref{algo:AlternatingSampler} returns a random point $x_T$ that satisfies
\begin{align}
\chi_\nu^2 (\mu_T) \le 
\frac{\chi_\nu^2 (\mu_0) }{(1 + C_\PI \eta )^{2T}}
\le \exp( \delta^2 /
8
) -1
\end{align}
in 
$T \ge \frac{1}{2} {\log \left( 
\frac{\chi_\nu^2(\mu_0)}{\exp( \delta^2 /8 ) -1  }
 \right)}/{\log(1+ C_\PI \eta)} 
 = \cO \left( {\log \left( 
\frac{\chi_\nu^2(\mu_0)}{ \delta^2 }
 \right)}/{\log(1+ C_\PI \eta)} \right) $  steps. Thus $\|\mu_T - \nu\| \le \delta /2$.

Putting two together, we have
\begin{align}
\| \hat \mu_T - \nu \|_\TV  \le   \| \hat \mu_T - \mu_T \|_\TV  + \| \mu_T - \nu \|_\TV  \le \delta .
\end{align}
Since $ \eta C_\PI = \cO(1)$, we have $\log(1+ \eta C_\PI ) = \cO( \eta C_\PI )$.
Thus, plugging in the value of $\eta$, we need 
$$T =  \cO \left( 
 \frac{L_\alpha^{\frac{2}{\alpha+1}}d^{\frac{\alpha}{\alpha+1}} }{C_\PI}\log \left(\frac{L_\alpha^{\frac{2}{\alpha+1}} d^{\frac{\alpha}{\alpha+1}} 
}{C_\PI \delta } \right) \log \left( \frac{\chi_\nu^2(\mu_0)}{\delta^2} \right) \right) $$ steps. Each step accesses only $\cO(1)$ many $f(x)$ in expectation because of Theorem \ref{thm:rgo}. 
\end{proof}

\subsection{Proof of Theorem \ref{thm:composite_rgo_tv}}

 \begin{proof}
Recall that the return of Algorithm \ref{alg:rgo} follows the distribution $ \hat \pi^{X|Y}$, and the target distribution of RGO is $\pi^{X|Y} \propto \exp(-g(x) - \frac{1}{2 \eta} \|x - x_y\|^2 ) $.
 Following the same proof of Theorem \ref{thm:rgo}, we have
$\| \pi^{X|Y} - \hat \pi^{X|Y} \|_\TV 
 \le \frac{2\mE |V - \overline V|}{ |\mE V|},$
 $|\mE V| \ge 1$, and $
 \mE[V- \overline V ] 
       \le       \mE[\rho \one_{\rho \ge 2} ]$.

Denote $\Delta = g(z) - g(x)$ and $\bar \Delta = \Delta / \log 2$, we have
\begin{align}
\mE[\rho \one_{\rho \ge 2} ] 
& =
\mE[\exp( \Delta) \one_{\rho \ge 2} ] = \mE[\exp( \Delta) \one_{\Delta \ge \log 2 } ]
= \mE[\exp( \bar \Delta \log 2) \one_{ \bar \Delta \ge 1 } ]  \\
& \le \sum_{i=1}^\infty \exp( (i+1)\log 2) \Pr(\bar \Delta \ge i) .
\end{align}
Note that $g(x) = f(x) - \<  f'(x_y), x\>$ satisfies that $ g'(x_y) = \sum_{i=j}^n  (f'_j(x_y) -  f'_j(x_y) )=  0 $, and $\mE[x] = x_y$. Moreover, $g(x)$ also satisfies \eqref{eq:composite}
because $ g'_j(x_1) -  g'_j(x_2) =  f'_j(x_1) -  f'_j(x_2)  $ for any $x_1,x_2$.
We also have the inequality $(1-0.5/a)^{-a/2} \le 1.5$  when $a \ge 1$.
We plug $\epsilon = 0.5 $ in Corollary \ref{cor:composite_concentration}, and obtain that, for $\forall r >0$,
\begin{align}
\Pr[g(x) - \mE g(x) \ge r ] \le \frac{3}{2}
\sum_{j=1}^n
\exp \left(  - \frac{
\Cjepshalf
(w_jr)^{\frac{2}{1+{\alpha_j}}}
 }{
 L_{\alpha_j}^{\frac{2}{{\alpha_j}+1}}d^{\frac{{\alpha_j}}{{\alpha_j}+1}}\eta
 } 
 \right) .
\end{align}
This further implies that 
\begin{align}
\Pr[g(z) - g(x) \ge r ] 
& = \Pr[g(z) - \mE g(z) + \mE g(x) -  g(x) \ge r ] \\
& \le \Pr \left[g(z) - \mE g(z) \ge \frac{r}{2} \right]+ \Pr \left[\mE g(x) -  g(x) \ge \frac{r}{2} \right] \\
& = \Pr \left[g(z) - \mE g(z) \ge \frac{r}{2} \right]+ \Pr \left[ -  g(x) - \mE [-g(x)] \ge \frac{r}{2} \right] \\
& \le 3 \sum_{j=1}^n \exp \left( - \frac{
\Cjepshalf
\left(\frac{w_j r}{2}\right)^{\frac{2}{1+{\alpha_j}}}
 }{
 L_{\alpha_j}^{\frac{2}{{\alpha_j}+1}}d^{\frac{{\alpha_j}}{{\alpha_j}+1}} 
 \eta } 
 \right) .
\end{align}
Thus 
$
\Pr[\bar \Delta \ge i ] \le 
3 \sum_{j=1}^n \exp \left( -\frac{
\Cjepshalf
\left(\frac{w_j i \log 2}{2}\right)^{\frac{2}{1+{\alpha_j}}}
 }{
 L_{\alpha_j}^{\frac{2}{{\alpha_j}+1}}d^{\frac{{\alpha_j}}{{\alpha_j}+1}} 
 \eta } 
 \right) . $
 We choose
\begin{align}\label{eq:weight}
     w_j = \frac{L_{\alpha_j}^{\frac{1}{\alpha_j+1}} d^{ \frac{\alpha_j}{ 2(\alpha_j+1)} }}{ \sum_{j=1}^n L_{\alpha_j}^{\frac{1}{\alpha_j+1}} d^{ \frac{\alpha_j}{ 2(\alpha_j+1) } } } .
 \end{align}
 Since 
 $i^{\frac{2}{1+\alpha}} \ge i$ for any $i \ge 1$, $0 \le \alpha \le 1$, and the bound \eqref{eq:alpha_bound}, we have that 
 \begin{align}
& \sum_{i=1}^\infty  \exp( (i+1)\log 2) \Pr(\bar \Delta \ge i) \\
\le &  6 \sum_{i=1}^\infty \sum_{j=1}^n
\exp \left( i - \frac{
\Cjepshalf
\left(\frac{w_j i \log 2}{2}\right)^{\frac{2}{1+{\alpha_j}}}
 }{
 L_{\alpha_j}^{\frac{2}{{\alpha_j}+1}}d^{\frac{{\alpha_j}}{{\alpha_j}+1}}
 \eta } 
 \right) \\
 \le &  6 \sum_{i=1}^\infty \sum_{j=1}^n
\exp \left( i - \frac{
w_j^2 i
}{ \Consteta
 L_{\alpha_j}^{\frac{2}{{\alpha_j}+1}}d^{\frac{{\alpha_j}}{{\alpha_j}+1}}
 \eta } 
 \right) 
 \overset{\eqref{eq:weight}}{=} 6 n \sum_{i=1}^\infty 
\exp \left( i - \frac{ i
}{\Consteta
 \left( \sum_{j=1}^n L_{\alpha_j}^{\frac{1}{\alpha_j+1}} d^{ \frac{\alpha_j}{ 2(\alpha_j+1)} } \right)^2
 \eta } 
 \right) 
 \\
 = & {6n} /\left({\exp \left(\frac{ 1
}{\Consteta
 \left( \sum_{j=1}^n L_{\alpha_j}^{\frac{1}{\alpha_j+1}} d^{ \frac{\alpha_j}{ 2(\alpha_j+1)} } \right)^2
 \eta } -1 \right) -1 } \right)
 \le \frac{\zeta }{2} \label{eq:bound_E_V_bar_V_composite}
 \end{align}
 by the choice 
$$\eta \le \frac{1}{\Consteta  \left( \sum_{j=1}^n L_{\alpha_j}^{\frac{1}{\alpha_j+1}} d^{ \frac{\alpha_j}{ 2(\alpha_j+1)} } \right)^2(1+ \log (1+ 12 n /\zeta )) } .$$ 
Finally, by Lemma \ref{lem:dist_relation},
the acceptance probability is $\frac{1}{2} \mE[\overline V] $. Since $\zeta< 1 $, we have 
\begin{align}
 \mE[\overline V] \ge \mE[ V] -  \mE|V - \overline V| \ge  1- \frac{\zeta}{2} \ge \frac{1}{2} .   
\end{align}
Thus, the expected number of the iterations for the rejection sampling step is $\cO(1)$.
\end{proof}

\subsection{Proof of Theorem \ref{thm:composite_rgo_w2}}

 \begin{proof}
 It suffices to prove the claim:
\textit{If the step size 
\begin{align}\label{eq:w2_eta_composite}
\eta \le \min \left(  \frac{1}{
 \Consteta  \left( \sum_{j=1}^n L_{\alpha_j}^{\frac{1}{\alpha_j+1}} d^{ \frac{\alpha_j}{ 2(\alpha_j+1)} } \right)^2
(2+ \log( 1+ {192n  (d^2 + 2d)}/{ \zeta^2} )) } , 1\right)  ,  
\end{align}
then Algorithm \ref{alg:rgo} returns a random point $x$ that has $\zeta$ \textbf{squared}  Wasserstein-2 distance to the distribution $\pi^{X|Y}$.  Furthermore, if ~$0< \zeta <8d $, then the algorithm access $\cO(1)$ many $f(x)$ in expectation.}

Recall that the return of Algorithm \ref{alg:rgo} follows the distribution $ \hat \pi^{X|Y}$, and the target distribution of RGO is $\pi^{X|Y} \propto \exp(-g(x) - \frac{1}{2 \eta} \|x - x_y\|^2 ) $.
 Following the same proof of Theorem \ref{thm:rgo_w2}, we have 
\begin{align}
    W_2^2( \pi^{X|Y} , \hat \pi^{X|Y} ) 
    & \le 2 \left( \mE  \|x - x_y\|^4 \mE |V - \overline V|^2  \right)^\frac{1}{2} + \frac{ 2 \mE |V - \overline V| \cdot \mE [\|x- x_y\|^2 \overline V ]}{  \mE  \overline V } , \label{eq:w2_2_terms_composite} \\
    \mE  \|x - x_y\|^4 
    & \le   d^2 + 2d, \label{eq:fourth_moment} \\
    \mE |V - \overline V|^2 
    & \le \sum_{i=1}^\infty \exp( 2 (i+1)\log 2) \Pr(\bar \Delta \ge i), \label{eq:w2_v_2} \\
    \mE [\|x- x_y\|^2 \overline V ] 
& \le 
2 \eta d. \label{eq:w2_2_moment}
 \end{align}
Applying the same concentration inequality and weight \eqref{eq:weight} in the proof of Theorem \ref{thm:composite_rgo_tv}, we have
\begin{align}
&    \sum_{i=1}^\infty \exp( 2(i+1)\log 2) \Pr(\bar \Delta \ge i) \\
 \le & 12n /\left({\exp \left(\frac{ 1
}{\Consteta
 \left( \sum_{j=1}^n L_{\alpha_j}^{\frac{1}{\alpha_j+1}} d^{ \frac{\alpha_j}{ 2(\alpha_j+1)} } \right)^2
 \eta } - 2 \right) -1 } \right) 
\overset{\eqref{eq:w2_eta_composite}}{\le} \frac{\zeta^2}{ 
 16 (d^2 + 2d) }.
\end{align}
 This, together with the inequalities 
 \eqref{eq:fourth_moment}, \eqref{eq:w2_v_2},
 gives that 
 \begin{align}\label{eq:w2_term1_composite}
  2 \left( \mE  \|x - x_y\|^4 \mE |V - \overline V|^2  \right)^\frac{1}{2}  \le \frac{\zeta}{4} .   
 \end{align}
  Now we bound the second term in \eqref{eq:w2_2_terms_composite}.
  By the assumption \eqref{eq:w2_eta_composite},  we have
\begin{align}
      \eta \le \frac{1}{\Consteta \left( \sum_{j=1}^n L_{\alpha_j}^{\frac{1}{\alpha_j+1}} d^{ \frac{\alpha_j}{ 2(\alpha_j+1)} } \right)^2 (1+ \log (1+ 
 96 n d / \zeta
 )) } .
 \end{align}
  Thus, follow the same steps in bounding
 \eqref{eq:bound_E_V_bar_V_composite} and \eqref{eq:low_bound_bar_V}, we obtain 
\begin{align}
    \mE |V - \overline V| \le
    \frac{\zeta}{ 16 d}  ~~\text{and}~~
    \mE[ \overline V] \ge \frac{1}{2}.
\end{align}
By the choice $\eta \le 1$, the above bounds, and the inequalities \eqref{eq:w2_2_terms_composite}, \eqref{eq:w2_2_moment}, we reach the result $W_2^2( \pi^{X|Y} , \hat \pi^{X|Y} ) \le \zeta .$ 

The expectation of iterations in rejection sampling also follows the proof of Theorem \ref{thm:rgo_w2}.
\end{proof}

\section{RGO with approximate proximal optimization error}\label{sec:approximate_opt}

In Algorithm \ref{alg:rgo} of Section \ref{sec:exact_rgo}, we assume that we can solve the step \ref{step:proximal_opt} exactly, which is to solve the stationary point of $f_y^\eta (x) = f(x) + \frac{1}{2 \eta} \|x-y\|^2$.
However, it could be challenging to solve this optimization problem exactly when $f$ is non-convex. 
In this section, we tackle this issue and present the complexity analysis when we can only obtain an approximate stationary point of it.

\subsection{Algorithm}

\begin{algorithm2e}

\caption{Rejection sampling implementation of RGO with proximal optimization error \label{alg:rgo_appro}}

\textbf{Input}: $L_\alpha$-$\alpha$-semi-smooth function $f(x)$, step size $\eta>0$, current
point $y$

Compute an approximation solution $x_y$ satisfying 
\eqref{eq:stationary_error_tv} 
 by using Algorithm 3 in \citet{Liang2022APA}. Denote $g(x) = f(x) - \< f'(x_y)) , x \> ,~ w = y - \eta f'(x_y) $. \label{step:opt_inacc}

\Repeat{$u\leq\frac{1}{2}\rho$}{

Sample $x,z$ from the distribution $ 
\phi (\cdot) 
\propto\exp( -\frac{1}{2\eta}\|\cdot - w\|^{2}_2)$

$\rho = \exp(g(z) - g(x))$

Sample $u$ uniformly from $[0,1]$.
}
\textbf{Return} $x$

\end{algorithm2e}

We will use Algorithm \ref{alg:rgo_appro} as the RGO algorithm, and Lemma \ref{lem:dist_relation} still holds for this algorithm.
We resort to 
Nesterov's accelerated gradient descent method
to compute $w$ that is an $s$-approximation to $x_y$, i.e., $\|x_y - w\| \le s$, see \cite[Algorithm 3]{Liang2022APA}. And we have the following lemma to guarantee a similar equivalence as in Lemma \ref{lem:equi_rgo}.

\begin{lemma}\label{lem:non_exact_equi_rgo}
For any $x_y$, sampling from $ \pi^{X|Y} (x|y) \propto \exp(-f (x)-\frac{1}{2\eta}\|x-y\|^{2})$ is equivalent to sampling from distribution $\propto \exp (-  g(x) - \frac{1}{2 \eta} \|x - w \|^2 )$, where $g(x) = f(x) - \< f'(x_y) , x \> $ and $w = y - \eta f'(x_y) $.
If we further assume  $x_y $ is an approximate stationary point to $f(x) + \frac{1}{2\eta} \|x-y\|^2$, i.e., 
\begin{align}\label{eq:approx_stationary}
\left \| f'(x_y) + \frac{1}{\eta } ( x_y - y) \right \|  \le  \frac{s}{\eta} ,
\end{align}
 then we have $\|x_y - w\| \le s.$
\end{lemma}
\begin{proof}
We first show the potentials of two distributions are the same up to a constant.
\begin{align}
g(x) + \frac{1}{2 \eta} \|x-w \|^2 
& = g(x) + \frac{1}{2 \eta} \|x- y + \eta f'(x_y) \|^2  \\
& = g(x) + \frac{1}{2 \eta} \|x- y \|^2 + 
 \< x,  f'(x_y) \>
-
\< y,  f'(x_y) \>
+ \frac{\eta}{2} \|f'(x_y)\|^2 \\
& =  g(x) + \frac{1}{2 \eta} \|x- y \|^2 + 
 \< x,  f'(x_y) \> + \text{constant} \\
 & =  f (x)+\frac{1}{2\eta}\|x-y\|^{2}+ \text{constant} .
\end{align}
We use $f(x) = g(x) + \< f'(x_y) , x \> $ in the last equality. By the definition of $g(x)$, we have that $g'(x_y) =0$, thus
\begin{align}
 \|x_y -w \| = \|x_y - y + \eta f'(x_y) \| = \eta \left \|f'(x_y) + \frac{1}{\eta} (x_y -y) \right \| \overset{\eqref{eq:approx_stationary}}{\le} s.
\end{align}
\end{proof}
\subsection{Complexity analysis}
We first investigate the complexity of the optimization algorithm to reach a small enough tolerance. Later, it turns out that the tolerance in \eqref{eq:stationary_error_tv} will be enough.

\begin{lemma}\label{lem:opt_comp}
    Assume $\eta 
    =
    {1}/({\TwiceConsteta 
    L_\alpha^{\frac{2}{\alpha+1}}d^{\frac{\alpha}{\alpha+1}}(1+ \log (1+ 12/\zeta )) 
    })
    $. 
    Let $x_y \in \mR^d$ be an approximate stationary point of $f_y^\eta$, i.e.,
    \begin{align}\label{eq:stationary_error_tv}
       \left \| f'(x_y) + \frac{1}{\eta } ( x_y - y) \right \|   \le
       \frac{d^{\frac{1}{2(1+\alpha)}} }{\RtConsteta L_\alpha^{ \frac{1}{1+\alpha}} \eta}. 
    \end{align}
Then, the iteration complexity to find $x_y$ by \citet[Algorithm 3]{Liang2022APA} is $\tilde \cO(1)$.
\end{lemma}
\begin{proof}
According to \citet[Lemma A.2]{Liang2022APA}, $f_y^\eta := f(x) + \frac{1}{2 \eta} \|x-y\|^2 $ satisfies that
\begin{align}
\frac{\beta}{2}\|u-v\|^2-\theta \leq f_y^\eta (u)-f_y^\eta(v)-\left\langle (f_y^\eta)^{\prime}(v), u-v\right\rangle \leq \frac{L}{2}\|u-v\|^2+\theta, \quad \forall u, v \in \mathbb{R}^d,
\end{align}
with
\begin{align}
M = \frac{L_\alpha^{\frac{2}{1+\alpha}}}{ (1+\alpha)^{ \frac{1-\alpha}{1+\alpha} } } ,~~
 \beta = \frac{1}{\eta} - M,
 ~~  L = \frac{1}{\eta} + M,~~ \theta = \frac{1-\alpha }{2} .
\end{align}
Denote the upper bound in \eqref{eq:stationary_error_tv} as $\rho: = 
\frac{d^{\frac{1}{2(1+\alpha)}} }{\RtConsteta L_\alpha^{ \frac{1}{1+\alpha}} \eta}
$.
Since $(1+\alpha)^{ \frac{1-\alpha}{1+\alpha} } \ge 1$,
by simple calculation,
we can verify that
\begin{align}
2\sqrt{2} (\beta + L)\theta /\sqrt{\beta}  
=
 \frac{2 \sqrt{2} (1-\alpha) }{\eta \sqrt{ \frac{1}{\eta} - M } }
 \le 
  \frac{2 \sqrt{2} (1-\alpha) }{  L_\alpha^{\frac{1}{\alpha+1}} \eta \sqrt{ 
  \TwiceConsteta 
   d^{\frac{\alpha}{\alpha+1}}(1+ \log (1+ 12/\zeta )) 
  - 1 } } \le \rho.
\end{align}
Then by \citet[Lemma B.4]{Liang2022APA}, the number of iterations to obtain \eqref{eq:stationary_error_tv} is at most 
\begin{align}
\frac{2 \sqrt{L}+\sqrt{\beta}}{2 \sqrt{\beta}} \log \left(\frac{(\beta+L)^2 \|x^{(0)} - x^*\|^2}{\rho^2} \frac{2 \sqrt{L}+\sqrt{\beta}}{2 \sqrt{\beta}}+1\right) = \tilde\cO(1),   
\end{align}
where $x^{(0)}$ is the initialization point in the optimization algorithm and $x^*$ is a ground truth stationary point of $f_y^\eta$.
\end{proof}

Considering the error in solving the stationary point of $f_y^\eta$, we will need another concentration bound for the sampling complexity proof. 
Compared with Theorem \ref{thm:concentration}, this concentration inequality only has an additional coefficient $\exp \left(\frac{ s^2}{2\eta (d/\epsilon -1 ) } \right)$ with all other terms intact.

\begin{proposition}
\label{prop:errored_concentration}
Let $X \sim \cN(m_0,\eta\ide)$ be a Gaussian random variable in $\Rd$, and let $\ell$ be an $L_\alpha$-$\alpha$-semi-smooth function. Assume $ \ell'(m_1) = 0$, and $\|m_0 - m_1\| \le s $. Then for any $r>0, ~0 \le \alpha \le 1$, one has $\forall \epsilon \in (0,d)$,
\begin{align}
\Pr( \ell(X)-\mE(\ell(X))\geq r) 
\leq 
\exp \left(\frac{ s^2}{2\eta (d/\epsilon -1 ) } \right)
\left(1-\frac{\epsilon}{d} \right)^{-d/2} \exp\left(-\frac{C\epsilon^{\frac{\alpha}{1+\alpha}}r^{\frac{2}{1+\alpha}}}{L_\alpha^{\frac{2}{1+\alpha}}d^{\frac{\alpha}{1+\alpha}}\eta}\right),
\end{align}
where 
\begin{align}
C = (1+\alpha)\left(\frac{1}{\alpha}\right)^{\frac{\alpha}{1+\alpha}}\left(\frac{1}{\pi^2}\right)^{\frac{1}{1+\alpha}}2^{\frac{1-\alpha}{1+\alpha}}.
\end{align}
\end{proposition}

\begin{proof}
With the same procedure in the proof of Theorem \ref{thm:concentration}, one can obtain 
\begin{align}
\Pr(\ell(X)-\mE(\ell(X))\geq r) 
& \leq  \inf_{\lambda >0} \frac{\mE_{Z}\exp(\frac{\pi^2}{8}\eta\lambda^2\|\nabla \ell(Z)\|_2^2)}{\exp(\lambda r)}.
\end{align}
where $Z \sim \cN(0,\eta\ide)$ and $\nabla \ell(m_1-m_0) = 0$. In what follows, we denote $m_1-m_0$ by $\iota$.
Again, we proceed by considering three cases.
\begin{enumerate}
\item $\alpha = 0$

Notice that $\|\nabla \ell(Z)\|_2^2 \leq L_0^2$. By taking $\lambda = \frac{4r}{\pi^2\eta L_0^2}$, one could get
\begin{align}
\Pr(\ell(X)-\mE(\ell(X))\geq r) \leq  \exp(- \frac{2 r^2}{\pi^2 \eta L_0^2}).
\end{align}
Note that for the case that $\alpha=0$, the result is the same as the one in Theorem \ref{thm:concentration}.
\item $\alpha = 1$

Assume $-\frac{1}{2\eta} + \frac{\pi^2}{8}L_1^2\eta\lambda^2 <0$, then
\begin{align}
\frac{\mE_{Z}\exp(\frac{\pi^2}{8}\lambda^2\eta\|\nabla \ell(Z)\|_2^2)}{\exp(\lambda r)} 
& \leq \frac{\mE_{Z}\exp(\frac{\pi^2}{8} L_1^2\eta\lambda^2\|Z-\iota\|_2^2)}{\exp(\lambda r)} \\
& = \left(\frac{1}{1-\frac{\pi^2}{4}L_1^2\eta^2\lambda^2}\right)^{d/2}\exp\left(\frac{\nicefrac{1}{2\eta}}{\frac{4}{\pi^2L_1^2\eta^2\lambda^2}-1}\|
\iota\|^2_2\right)\exp(-\lambda r).
\end{align}
Let $\lambda = \frac{k}{L_1^2d\eta^2}$ and denote $\frac{\pi^2k^2}{4\eta^2L_1^2d}$ by $\epsilon$. Then,
\begin{align}
\Pr(f(X)-\mE(f(X))\geq r) \leq  (1-\frac{\epsilon}{d})^{-\frac{d}{2}}\exp\left(\frac{\nicefrac{1}{2\eta}}{\nicefrac{d}{\epsilon}-1}\|
\iota\|^2_2\right)\exp(-\sqrt{\frac{4\epsilon}{\pi^2}}\frac{r}{\eta L_1d^{1/2}}), \: \forall \epsilon \in (0,d)
\end{align}
Note that given the value of $\epsilon$, $(1-\frac{\epsilon}{d})^{-\frac{d}{2}}$ is bounded for $d$. One can also observe that compared with the result in Theorem \ref{thm:concentration}, the only difference is the additional term $\exp\left(\frac{\nicefrac{1}{2\eta}}{\nicefrac{d}{\epsilon}-1}\|
\iota\|^2_2\right)$.

\item $0 <\alpha <1$

By Young's inequality, for any $\omega>0$, one obtains
$\|Z-\iota\|^{2\alpha}\omega \leq \alpha\|Z-\iota\|^2 + (1-\alpha)\omega^{\frac{1}{1-\alpha}} $. Hence, with the assumption $1-\frac{\pi^2}{4}L_{\alpha}^2\eta^2\lambda^2\frac{\alpha}{\omega}>0$,
\begin{align}
& \frac{\mE_{Z}\exp(\frac{\pi^2}{8}\eta\lambda^2\|\nabla \ell(Z)\|_2^2)}{\exp(\lambda r)} 
\leq \frac{\mE_{Z}\exp(\frac{\pi^2}{8} L_{\alpha}^2\eta\lambda^2\|Z-\iota\|_2^{2\alpha})}{\exp(\lambda r)} \\
 \leq & (1-\frac{\pi^2}{4}\eta^2\lambda^2L_\alpha^2\frac{\alpha}{\omega})^{-\frac{d}{2}}
 \exp\left(\frac{\nicefrac{1}{2\eta} \|
\iota\|^2_2 }{\nicefrac{1}{\left(\frac{\pi^2}{4}\eta^2\lambda^2L_\alpha^2\frac{\alpha}{\omega}\right)}-1}\right)
\exp(\frac{\pi^2}{8}\eta\lambda^2L_\alpha^2(1-\alpha)\omega^{\frac{\alpha}{1-\alpha}})\exp(-\lambda r). \label{eq:young_split_appro}
\end{align}
One can notice that compared with the case $\alpha \in (0,1)$ in the proof of Theorem \ref{thm:concentration}, the only additional term is $\exp\left(\frac{\nicefrac{1}{2\eta}}{\nicefrac{1}{\left(\frac{\pi^2}{4}\eta^2\lambda^2L_\alpha^2\frac{\alpha}{\omega}\right)}-1}\|
\iota\|^2_2\right)$. With the same suboptimal  $\hat{\lambda}$ and $\hat{\omega}$ as in the proof of Theorem \ref{thm:concentration}, we have $\forall \epsilon \in (0,d)$,
\begin{equation}
    \Pr(\ell(X)-\mE(\ell(X))\geq r) 
    \leq \left(1-\frac{\epsilon}{d} \right)^{-d/2}\exp\left(\frac{\nicefrac{1}{2\eta}}{\nicefrac{d}{\epsilon}-1}\|
\iota\|^2_2\right)\exp\left(-\frac{C\epsilon^{\frac{\alpha}{\alpha+1}}r^{\frac{2}{\alpha+1}}}{L_\alpha^{\frac{2}{\alpha+1}}d^{\frac{\alpha}{\alpha+1}}\eta}\right),
\end{equation}
with $C$ being $(1+\alpha)\left(\frac{1}{\alpha}\right)^{\frac{\alpha}{1+\alpha}}\left(\frac{1}{\pi^2}\right)^{\frac{1}{1+\alpha}}2^{\frac{1-\alpha}{1+\alpha}}$.

\end{enumerate}

\end{proof}

Following the same procedures in Section \ref{sec:exact_rgo}, with the concentration inequality at our disposal, we can probabilistically bound the difference of $\rho$ and $\bar \rho$, which determines the discrepancy between $\hat \pi^{X|Y} $ and $\pi^{X|Y}$.

 \begin{theorem}[RGO complexity in total variation]
 \label{thm:rgo_tv_appro}
If the step size $$\eta \le \frac{1}{\TwiceConsteta  L_\alpha^{\frac{2}{\alpha+1}}d^{\frac{\alpha}{\alpha+1}}(1+ \log (1+ 12/\zeta )) } ,$$ 
then for any $\zeta>0$, Algorithm \ref{alg:rgo_appro} returns a random point $x$ that has $\zeta$ total variation distance to the distribution proportional to $\pi^{X|Y}(\cdot | y)$.
Furthermore, if $~0 < \zeta < 1$, then the algorithm access $\cO(1)$ many $f(x)$ in expectation.
 \end{theorem}
\begin{proof}
Following the same proof as Theorem \ref{thm:rgo}, we obtain 
\begin{align}
 \| \pi^{X|Y} - \hat \pi^{X|Y} \|_\TV \le  2 \sum_{i=1}^\infty \exp( (i+1)\log 2) \Pr(\bar \Delta \ge i) .
\end{align}
Note that $g(x) = f(x) - \< \nabla f(x_y), x\>$ satisfies that $\nabla g(x_y) =0 $, and 
$\mE[x] = w$. Moreover, $g(x)$ is also $L_\alpha$-$\alpha$-semi-smooth because $\nabla g(x_1) - \nabla g(x_2) = \nabla f(x_1) - \nabla f(x_2)  $ for any $x_1,x_2$.
We also have the inequality $(1-0.5/a)^{-a/2} \le 1.5$ and $0< 1/(2a-1) \le 1/a $   when $a \ge 1$.
Since \eqref{eq:stationary_error_tv} in Lemma \ref{lem:opt_comp} is satisfied, further by Lemma \ref{lem:non_exact_equi_rgo}, we have 
\begin{align}\label{eq:real_s_tv}
\|x_y - w\| \le s := \frac{d^{\frac{1}{2(1+\alpha)}} }{\RtConsteta L_\alpha^{ \frac{1}{1+\alpha}} } .    
\end{align}
Thus, we plug in $\epsilon = 0.5 $ in Proposition \ref{prop:errored_concentration}, and obtain that, for $\forall r >0$,
\begin{align}\label{eq:const_concentrate_appro}
\Pr[g(x) - \mE g(x) \ge r ] \le \frac{3}{2} \exp \left( \frac{s^2 }{ 2 d \eta}  - \frac{
\Cepshalf
r^{\frac{2}{1+\alpha}}
 }{
 L_\alpha^{\frac{2}{\alpha+1}}d^{\frac{\alpha}{\alpha+1}}\eta
 } 
 \right) .
\end{align}
This further implies that 
\begin{align}
\Pr[g(z) - g(x) \ge r ] 
& = \Pr[g(z) - \mE g(z) + \mE g(x) -  g(x) \ge r ] \\
& \le \Pr \left[g(z) - \mE g(z) \ge \frac{r}{2} \right]+ \Pr \left[\mE g(x) -  g(x) \ge \frac{r}{2} \right] \\
& = \Pr \left[g(z) - \mE g(z) \ge \frac{r}{2} \right]+ \Pr \left[ -  g(x) - \mE [-g(x)] \ge \frac{r}{2} \right] \\
& \le 3 \exp \left( - \left( \frac{
\Cepshalf
\left(\frac{r}{2}\right)^{\frac{2}{1+\alpha}}
 }{
 L_\alpha^{\frac{2}{\alpha+1}}d^{\frac{\alpha}{\alpha+1}} 
 } - \frac{s^2}{2 d} \right) \frac{1}{\eta}
 \right) .
\end{align}
Thus 
$
\Pr[\bar \Delta \ge i ] \le 
3 \exp \left( - \left( \frac{
\Cepshalf
\left(\frac{i \log 2}{2}\right)^{\frac{2}{1+\alpha}}
 }{
 L_\alpha^{\frac{2}{\alpha+1}}d^{\frac{\alpha}{\alpha+1}} 
 } - \frac{s^2}{2 d } \right) \frac{1}{\eta}
 \right) . $
Denote $$C_\eta : =  \left( 
\frac{
\Ceta
}{
  L_\alpha^{\frac{2}{\alpha+1}}d^{\frac{\alpha}{\alpha+1}}
 } - \frac{s^2}{2 d} \right) \frac{1}{\eta} -1 . $$ 
  Since $\log 2 < 1$ and  $i^{\frac{2}{1+\alpha}} \ge i$ for any $i \ge 1$, $0 \le \alpha \le 1$, we have that 
 \begin{align}
& \sum_{i=1}^\infty \exp( (i+1)\log 2) \Pr(\bar \Delta \ge i) \\
\le & 6 \sum_{i=1}^\infty 
\exp \left( i - \left( \frac{
\Cepshalf
\left(\frac{i \log 2}{2}\right)^{\frac{2}{1+\alpha}}
 }{
 L_\alpha^{\frac{2}{\alpha+1}}d^{\frac{\alpha}{\alpha+1}}
 } - \frac{s^2}{2 d } \right) \frac{1}{\eta}
 \right) \\
 \le  & 6 \sum_{i=1}^\infty 
\exp \left(  - C_\eta
 i \right)  = \frac{6}{\exp(C_\eta) -1 } \le \frac{\zeta }{2} \label{eq:eq:bound_E_V_bar_V_appro}
 \end{align}
 by the choice 
 $$\eta \le \left( \frac{
\Ceta
}{
  L_\alpha^{\frac{2}{\alpha+1}}d^{\frac{\alpha}{\alpha+1}} }
 - \frac{s^2}{2 d }
 \right)
 \frac{1}{(1+ \log (1+ 12/\zeta ))} . $$
 Since the bound \eqref{eq:alpha_bound} and the value of $s
$ in \eqref{eq:real_s_tv}, 
the bound
$$\eta \le \frac{1}{\TwiceConsteta  L_\alpha^{\frac{2}{\alpha+1}}d^{\frac{\alpha}{\alpha+1}}(1+ \log (1+ 12/\zeta )) } $$ suffices.
This finishes the proof.

Finally, by Lemma \ref{lem:dist_relation},
the acceptance probability is $\frac{1}{2} \mE[\overline V] $. Since $\zeta< 1 $, we have 
\begin{align}
 \mE[\overline V] \ge \mE[ V] -  \mE|V - \overline V| \ge  1- \frac{\zeta}{2} \ge \frac{1}{2} .   
\end{align}
Thus, the expected number of iterations for the rejection sampling step is $\cO(1)$.
\end{proof}

 \begin{theorem}[RGO complexity in Wasserstein distance] \label{thm:rgo_w2_appro}
If the step size 
\begin{align}
\eta \le \min \left(  \frac{1}{\TwiceConsteta L_\alpha^{\frac{2}{\alpha+1}}d^{\frac{\alpha}{\alpha+1}}(2+ \log( 1+ {192  (d^2 + 2d)}/{ \zeta^4} )) } , 1\right)  ,  
\end{align}
then Algorithm \ref{alg:rgo_appro} returns a random point $x$ that has $\zeta$ Wasserstein-2 distance to the distribution $\pi^{X|Y}(\cdot | y)$.  Furthermore, if 
 ~$0< \zeta < 2\sqrt{2d} $, then the algorithm access $\cO(1)$ many $f(x)$ in expectation.
 \end{theorem}
 \begin{proof}
 It suffices to prove the claim:
\textit{If the step size 
\begin{align}\label{eq:w2_eta_appro}
\eta \le \min \left(  \frac{1}{\TwiceConsteta L_\alpha^{\frac{2}{\alpha+1}}d^{\frac{\alpha}{\alpha+1}}(2+ \log( 1+ {192  (d^2 + 2d)}/{ \zeta^2} )) } , 1\right)   ,  
\end{align}
then Algorithm \ref{alg:rgo} returns a random point $x$ that has $\zeta$ \textbf{squared}  Wasserstein-2 distance to the distribution $\pi^{X|Y}$.  Furthermore, if ~$0< \zeta <8d $, then the algorithm access $\cO(1)$ many $f(x)$ in expectation.}
 
  By Lemma \ref{lem:tv_w2} (\citet[Proposition 7.10]{villani2021topics}) and triangular inequality, 
    \begin{align}
W_2^2( \pi^{X|Y} , \hat \pi^{X|Y} ) & \le 2 \left\| \| \cdot - w \|_2^2 (\pi^{X|Y} - \hat \pi^{X|Y} ) \right\|_\TV 
= 2 \mE \left( \|x - w \|_2^2 \left| \frac{V}{\mE V} - \frac{\overline V}{\mE \overline V} \right| \right) \\
& \le 2  \mE \|x - w \|_2^2 \left( \left| \frac{V}{ \mE V } - \frac{\overline V}{ \mE  V}  \right| +  \left| \frac{\overline V}{ \mE  V } - \frac{\overline  V}{ \mE \overline V}  \right| \right)  \\
&\le \frac{ 2 \mE [ \|x - w\|^2 |V - \overline V| ] }{ \mE V } 
+  \frac{ 2 \mE [\|x- w\|^2 \overline V |\mE [V - \overline V]|]}{ \mE V  \mE  \overline V } \\
& \overset{\eqref{eq:bound_V}}{\le} 2 \mE [ \|x - w\|^2 |V - \overline V| ] 
+  \frac{ 2 \mE |V - \overline V| \cdot \mE [\|x- w\|^2 \overline V ]}{  \mE  \overline V } . 
    \end{align}
Firstly, by Cauchy-Schwartz inequality,
\begin{align}
 \mE [ \|x - w\|^2 |V - \overline V| ] \le \left( \mE  \|x - w\|^4 \mE |V - \overline V|^2  \right)^\frac{1}{2} .
\end{align}
Similar to \eqref{eq:4_moment} and \eqref{eq:weighted_bar_V}, we have 
\begin{align}\label{eq:w_part1_appro}
    \mE  \|x - w\|^4 \le    d^2 + 2d  ~~\text{and}~~ \mE [\|x- w\|^2 \overline V ] \le 2 \eta d.
\end{align}
On the other hand, following the same logic of bounding $\mE [V - \overline V]$ in Theorem \ref{thm:rgo},
\begin{align} \label{eq:v_diff_squ_appro}
\mE |V - \overline V|^2 
\le \sum_{i=1}^\infty \exp( 2 (i+1)\log 2) \Pr(\bar \Delta \ge i) .
\end{align}
Note that the $s$ value in \eqref{eq:real_s_tv} still holds.
By the concentration \eqref{eq:const_concentrate_appro}, and the bound \eqref{eq:alpha_bound},
we have 
\begin{align}
  &  \sum_{i=1}^\infty \exp( 2 (i+1)\log 2) \Pr(\bar \Delta \ge i) \\
   \overset{\eqref{eq:const_concentrate_appro}}{\le}  & 12 \sum_{i=1}^\infty 
\exp \left( 2 i - \left(\frac{
\Cepshalf
\left(\frac{i \log 2}{2}\right)^{\frac{2}{1+\alpha}}
 }{
 L_\alpha^{\frac{2}{\alpha+1}}d^{\frac{\alpha}{\alpha+1}} 
 } - \frac{s^2}{ 2d } \right) \frac{1}{\eta}
 \right)
 \\
 \le & 12 \sum_{i=1}^\infty  \exp\left( 2 i - \left(\frac{ i
 }{\Consteta
 L_\alpha^{\frac{2}{\alpha+1}}d^{\frac{\alpha}{\alpha+1}} 
 } - \frac{s^2}{ 2d } \right) \frac{1}{\eta}
 \right) 
 \le  12 \exp  \left( 2 i - \left(\frac{ 1
 }{\Consteta
 L_\alpha^{\frac{2}{\alpha+1}}d^{\frac{\alpha}{\alpha+1}} 
 } - \frac{s^2}{ 2d } \right) \frac{i}{\eta}
 \right) \\
=&   \frac{12}{\exp \left(\left(\frac{ 1
 }{\Consteta
 L_\alpha^{\frac{2}{\alpha+1}}d^{\frac{\alpha}{\alpha+1}} 
 } - \frac{s^2}{ 2d } \right) / \eta -2  \right) - 1 }\overset{\eqref{eq:real_s_tv}}{=}  \frac{12}{\exp 
 \left(\frac{ 1
 }{\TwiceConsteta 
 L_\alpha^{\frac{2}{\alpha+1}}d^{\frac{\alpha}{\alpha+1}} 
 \eta } 
  -2  \right) - 1 } 
 \overset{\eqref{eq:w2_eta_appro}}{\le}  \frac{\zeta^2}{ 
 16 (d^2 + 2d) }  .
\end{align}
This gives $  \mE [ \|x - w \|^2 |V - \overline V| ]  \le \frac{\zeta}{4} . $ Next, following the same steps to bound \eqref{eq:eq:bound_E_V_bar_V_appro} and \eqref{eq:low_bound_bar_V}, we obtain 
\begin{align}
\label{eq:V_and_difference_approx}
    \mE |V - \overline V| \le
    \frac{\zeta}{ 16 d}  ~~\text{and}~~
    \mE[ \overline V] \ge \frac{1}{2}.
\end{align}
The above three inequalities, together with \eqref{eq:w_part1_appro},   gives that $W_2^2( \pi^{X|Y} , \hat \pi^{X|Y}  ) \le \zeta .$ 

Finally, because of the choice $\zeta< 8d $, and the bound \eqref{eq:V_and_difference_approx}, we have  $  \mE |V - \overline V| 
\le
    \frac{1}{ 2}  $.
Thus, the acceptance probability $\frac{1}{2} \mE[\overline V]
 \ge \frac{1}{2}( \mE[ V] -  \mE|V - \overline V| )
\ge \frac{1}{4} $.  So the expectation of iterations in rejection sampling follows as $\cO(1)$.
 \end{proof}

These two theorems show that the step size will have the same order as in Section \ref{sec:exact_rgo}. Thus all the results in Section \ref{sec:convergence} also follow if the RGO is realized by Algorithm \ref{alg:rgo_appro}.

\section{Extension to the convergence in $\chi^2$-divergence}

We now extend our results in \S\ref{sec:exact_rgo}, \S\ref{sec:convergence}, \S\ref{sec:approximate_opt} to the strong notion of $\chi^2$-divergence. We do not modify the the concentration inequalities nor the RGO algorithm, but only the proof of RGO step size. The new RGO results in $\chi^2$ with both accurate and inaccurate optimization step are shown in \S\ref{sec:rgo_chi}. Then in \S \ref{sec:convegence_chi},  we combine our RGO result with the techniques in \cite{Altschuler2023} to get the final convergence results in $\chi^2$.

\subsection{RGO for semi-smooth potential} \label{sec:rgo_chi}

We first consider we can solve optimization step \ref{step:opt} accurately in Algorithm \ref{alg:rgo}.
This theorem is more general than Theorem \ref{thm:rgo} and \ref{thm:rgo_w2}, which are presented under the weaker metric TV or $W_2$.
\begin{theorem}[RGO complexity in $\chi^2$-divergence with accurate optimization step \ref{step:opt}]\label{thm:rgo_chi}
Assume $f$ satisfies \eqref{eq:semi-smooth}. For $\forall \zeta>0$,
if
$$\eta \le 
\left(49  L_\alpha^{\frac{2}{\alpha+1}}d^{\frac{\alpha}{\alpha+1}}(2+ \log (1+ 522/\zeta )) \right)^{-1}
,$$
then Algorithm \ref{alg:rgo} returns a random point $x$ that has $\zeta ~~\chi^2$-divergence  to the distribution proportional to $
\pi^{X|Y}(\cdot | y)
$. Furthermore, if $~0 < \zeta < 0.5 $, then the algorithm access $\cO(1)$ many $f(x)$ in expectation.
 \end{theorem}
\begin{proof}
Recall that the return of Algorithm \ref{alg:rgo} follows the distribution $ \hat \pi^{X|Y}$, and the target distribution of RGO is $\pi^{X|Y} \propto \exp(-g(x) - \frac{1}{2 \eta} \|x - x_y\|^2 ) $.
According to Lemma \ref{lem:dist_relation}, we have 
\begin{align}
\chi_{\pi^{X|Y}  }^2( \hat \pi^{X|Y} ) 
& = \frac{\mE[\overline V^2 V^{-1}] \mE[V]}{(\mE \overline V)^2} -1  
= \frac{\mE[(V-\overline V)^2 V^{-1}] \mE[V] - (\mE[V-\overline V] )^2 }{(\mE \overline V)^2 } \\
& \le \frac{\mE[(V-\overline V)^2 V^{-1}] \mE[V]}{(\mE \overline V)^2}
\end{align}
We first bound the two expectations in the numerator.
Denote $\Delta = g(z) - g(x)$.
Following \eqref{eq:exp_grad} in the proof of concentration inequality in Theorem \ref{thm:concentration}, we have 
\begin{align}
 \mE[V] 
 & = \mE_{x,z}[\rho]  = \mE_{x,z}[\exp(\Delta)]
 \le \mE_x[ \exp( \pi^2 \eta \|\nabla g(x)\|^2 /8 )] \\
& \overset{\eqref{eq:young_split}}{\le} 
\left(1- \frac{ \pi^2 L_\alpha^2 \eta^2 \alpha}{ 4w}\right)^{-\frac{d}{2}} \exp \left( \frac{ \pi^2}{8} L_\alpha^2 \eta(1-\alpha) w^{\frac{\alpha}{1-\alpha}} \right) 
\end{align}
for $\forall w>0$. By choosing $w = (1-\alpha)^{-\frac{1-\alpha}{\alpha}} (\eta d)^{1-\alpha}  $, the above term becomes
\begin{align}
 \left(1- \frac{ \pi^2 L_\alpha^2 \eta^{1+\alpha} \alpha (1-\alpha)^{ \frac{1-\alpha}{\alpha} } }{ 4 d^{1-\alpha}}\right)^{-\frac{d}{2}} 
 \exp \left( \frac{ \pi^2}{8} L_\alpha^2 \eta^{1 + \alpha} d^\alpha \right) .   
\end{align}
We further choose $\eta \le  \frac{1 }{2 L_\alpha^\frac{2}{1+\alpha} d^\frac{\alpha}{1+\alpha} } \le \left( \frac{8 \log(1.2) }{\pi^2 L_\alpha^2 d^\alpha } \right)^\frac{1}{1+\alpha} $, so the above term is smaller than
\begin{align}
1.2 \left(1- \frac{  2 \alpha (1-\alpha)^{ \frac{1-\alpha}{\alpha} } \log(1.2) }{ d }\right)^{-\frac{d}{2}} 
\le  1.2 \exp( 2 \alpha (1-\alpha)^{ \frac{1-\alpha}{\alpha} } \log(1.2) ) \le 1.2 \times 1.5 = 2.
\end{align}
We use the inequality $(1/(1-x))^{1/x} \le \exp(2)$ if $0 \le x \le 0.5$ above. So we obtain $\mE[V] \le 2$.

When $V=\mE[\rho|x]$ is small, $\rho$ is likely to be small and thus $V-\overline V = \mE [(\rho-2) \one_{\rho>2} ]$ is also small.
We bound the term $\mE[(V-\overline V)^2 V^{-1}] $ by splitting it into 
\begin{align}
\mE \frac{(V-\overline V)^2}{V} 
\le \mE \frac{(V-\overline V)^2}{V} \one_{V \le \frac{1}{e} } +  \mE \frac{(V-\overline V)^2}{V} \one_{V > \frac{1}{e}}
\le \mE \frac{(V-\overline V)^2}{V} \one_{V \le \frac{1}{e} } +  e \mE (V-\overline V)^2 .
\end{align}
Following the same logic of \eqref{eq:v_diff_squ} in the proof of Theorem \ref{thm:rgo_w2}, we can get that if $$\eta \le \frac{
1
}{
  49 L_\alpha^{\frac{2}{\alpha+1}}d^{\frac{\alpha}{\alpha+1}}
(2+ \log( 1+ \frac{522 }{ \zeta } )) }, $$
then $e\mE (V - \overline V)^2 \le \zeta/(16 )  . $ 
We then notice that 
when $V \le 1/e$, by the definition of $V$ and Jensen inequality, it holds that $g(x) - \mE [g] \ge 1. $ 
To bound the term $\mE \frac{(V-\overline V)^2}{V} \one_{V \le \frac{1}{e} }$, we write
\begin{align}
\mE \frac{(V-\overline V)^2}{V} \one_{V \le \frac{1}{e} }
\le \frac{1}{e} \mE \left(\frac{V-\overline V}{V}\right)^2 \one_{V \le \frac{1}{e} } 
\le \frac{1}{e} \mE \left(\frac{V-\overline V}{V}\right)^2 \one_{g(x) - \mE[g] \ge 1}.
\end{align}
Moreover, since $g(x) -\mE[g]$ 
is always no less than 1, with $c= 0.2 $, we have 
\begin{align}\label{eq:cond_conce}
\Pr (\Delta
> r |x) 
\le 1.5
\exp \left(-c \frac{(g(x) -\mE[g] + r)^\frac{2}{\alpha + 1} } {L_\alpha^{ \frac{2}{\alpha+1} } d^{\frac{\alpha}{\alpha+1}} \eta}  \right)  
\le  1.5
\exp \left(-c \frac{g(x) -\mE[g] + r } {L_\alpha^{ \frac{2}{\alpha+1} } d^{\frac{\alpha}{\alpha+1}} \eta}  \right)  
\end{align}
for $\forall r > 0 $.

Denote $\bar \Delta = \Delta / \log 2 =  (g(z) - g(x) )/\log 2$, and $S = L_\alpha^{ \frac{2}{\alpha+1} } d^{\frac{\alpha}{\alpha+1}} \eta $. We can write 
\begin{align}
 \frac{V-\overline V}{V}  = & \frac{\mE[(\rho-2) \one_{\rho > 2} |x ]}{ \mE[\exp(\Delta) |x ] } 
\le \frac{\mE[ \exp(\Delta) \one_{\rho > 2} |x ]}{ \mE[\exp(\Delta) |x ] } =  \frac{1}{  \mE[\exp(\Delta) \one_{\rho \le 2} |x ] /  \mE[\exp(\Delta) \one_{\rho > 2} |x ]  +1}    \\
\le & \frac{\mE[\exp(\Delta) \one_{\rho > 2} |x ] }{ \mE[\exp(\Delta) \one_{\rho \le 2} |x ]} 
\le \frac{2 \sum_{i=1}^\infty \exp(i \log 2) \Pr (i \le \bar \Delta \le i+1 | x) }{ \exp( \frac{-1 + \mE[g] - g(x)}{\log 2} )  \Pr ( \frac{-1 + \mE[g] - g(x)}{\log 2} \le \bar \Delta \le 1 | x) } \\
\le &  \frac{2 \sum_{i=1}^\infty \exp(i \log 2) \Pr ( \bar \Delta \ge i | x) }{ \exp( \frac{-1 + \mE[g] - g(x)}{\log 2} )  [1- \Pr (  \Delta > \log 2 | x) - \Pr (  g(z) - \mE[g] <-1 ) ] }  \label{eq:v_frac} \\
\overset{ \eqref{eq:cond_conce} }{\le} & \frac{ 3 \sum_{i=1}^\infty \exp(i \log 2) \exp( - c \frac{g(x) - \mE[g] + i\log 2
}{S} ) \exp( \frac{g(x) - \mE[g] }{\log 2} ) }{ \exp(- \frac{1}{\log 2}) (1- 1.5 [ \exp(- c \frac{g(x) - \mE[g] + \log 2
}{S}  ) + \exp(- \frac{c}{S} )]) } \\
\le & \frac{ 15 \sum_{i=1}^\infty \exp(i (\log 2- \frac{c \log2
}{S}) ) \exp( - (\frac{c}{S} - \frac{1}{\log 2}) (g(x) - \mE[g] )  )  }{ 1- 1.5 [ \exp(- c \frac{g(x) - \mE[g] + \log 2
}{S}  ) + \exp(- \frac{c}{S} )] } \\
= & \frac{ 15  \exp( - (\frac{c}{S} - \frac{1}{\log 2}) (g(x) - \mE[g] )  ) / ({\exp( \frac{c \log2
}{S} - \log 2 ) - 1}) }{ 1- 1.5 [ \exp(- c \frac{g(x) - \mE[g] + \log 2
}{S}  ) + \exp(- \frac{c}{S} )] } \\
\le & \frac{ 15  \exp( - (\frac{c}{S} - \frac{1}{\log 2}) 
) / ({\exp( \frac{c \log2
}{S} - \log 2 ) - 1}) }{ 1- 1.5 [ \exp(- c \frac{ 1
+ \log 2
}{S}  ) + \exp(- \frac{c}{S} )] } ,
\end{align}
where we use the condition $g(x) - \mE[g] \ge 1$ in the last inequality.
By choosing
\begin{align}
\eta \le  ({8 L_\alpha^{ \frac{2}{\alpha+1} } d^{\frac{\alpha}{\alpha+1}} \max \{1.5 + \log( 1 + 7 /\sqrt{\zeta}), \log(30 ) \} ) } )^{-1} ,   
\end{align}
under the condition $g(x) - \mE[g] \ge 1$,
 we guarantee that $1- 1.5 [ \exp(- c \frac{ 1
+ \log 2
}{S}  ) + \exp(- \frac{c}{S} ) \ge 0.8 $ and $ 15  \exp( - (\frac{c}{S} - \frac{1}{\log 2}) 
) / ({\exp( \frac{c \log2
}{S} - \log 2 ) - 1}) \lesssim \zeta $, thus $ (\frac{V- \overline V}{V})^2 \le  \frac{e \zeta}{16}. $   So we obtain that $\mE \frac{(V-\overline V)^2}{V} \one_{V \le \frac{1}{e} }
 \le \frac{\zeta}{16} $, and thus $\mE \frac{(V-\overline V)^2}{V}  \le \frac{\zeta}{8}. $

To bound $\mE \overline V$, we notice that if $\chi_{\pi^{X|Y}}^2(\hat \pi^{X|Y} ) \le 0.5, $ then $\|\hat \pi^{X|Y} - \pi^{X|Y} \|_\TV \le 1, $ thus by \eqref{eq:low_bound_bar_V}, we get $\mE \overline V \ge \frac{1}{2}.$ Combining it with the bounds $\mE[V] \le 2 $ and $\mE \frac{(V-\overline V)^2}{V}  \le \frac{\zeta}{8}, $ we get $\chi_{ \pi^{X|Y} }^2( \hat \pi^{X|Y} )  \le \zeta$.

 The expected number of the iterations for the rejection sampling step is $\cO(1)$ because
the acceptance probability is $\frac{1}{2} \mE[\overline V] $ (Lemma \ref{lem:dist_relation}).
\end{proof}

We then consider the case where we can only solve the optimization step \ref{step:opt_inacc} inaccurately. The corresponding algorithm and analysis have been discussed in \S\ref{sec:approximate_opt}.
The following theorem extends the results of Theorem \ref{thm:rgo_tv_appro} and \ref{thm:rgo_w2_appro} in \S\ref{sec:approximate_opt}.

 \begin{theorem}[RGO complexity in $\chi^2$-divergence]\label{thm:rgo_chi_opt}
Assume $f$ satisfies \eqref{eq:semi-smooth}. For $\forall \zeta>0$,
if
$$\eta \le 
\left(98  L_\alpha^{\frac{2}{\alpha+1}}d^{\frac{\alpha}{\alpha+1}}(2+ \log (1+ 522/\zeta )) \right)^{-1}
,$$
then Algorithm \ref{alg:rgo_appro} returns a random point $x$ that has $\zeta ~~\chi^2$-divergence  to the distribution proportional to $
\pi^{X|Y}(\cdot | y)
$. Furthermore, if $~0 < \zeta < 0.5 $, then the algorithm access $\cO(1)$ many $f(x)$ in expectation.
 \end{theorem}
 \begin{proof}
Following the proof with accurate optimization, we get 
\begin{align}
\chi_{\pi^{X|Y} }^2( \hat \pi^{X|Y}  ) 
 \le \frac{\mE[(V-\overline V)^2 V^{-1}] \mE[V]}{(\mE \overline V)^2}    
\end{align}
 When bounding $\mE V$, without loss of generality, we assume $ w $ is 0, i.e., the mean of $x$ and $z$ is 0, and $g'(x_y - w)=0$. 
 Denote $\Delta = g(z) - g(x)$, and $\iota = x_y - w $.
 Following \eqref{eq:exp_grad} in the proof of concentration inequality in Theorem \ref{thm:concentration}, we have 
\begin{align}
 \mE[V] 
 & = \mE_{x,z}[\rho]  = \mE_{x,z}[\exp(\Delta)]
 \le \mE_x[ \exp( \pi^2 \eta \|\nabla g(x)\|^2 /8 )] \\
& \overset{\eqref{eq:young_split_appro}}{\le} 
\left(1- \frac{ \pi^2 L_\alpha^2 \eta^2 \alpha}{ 4w}\right)^{-\frac{d}{2}} \exp \left( \frac{ \pi^2}{8} L_\alpha^2 \eta(1-\alpha) w^{\frac{\alpha}{1-\alpha}} \right) \exp\left(\frac{ \|
\iota\|^2_2 / (2\eta) }{{\left(\frac{\pi^2}{4}\eta^2 L_\alpha^2\frac{\alpha}{\omega}\right)}^{-1}-1}\right)
\end{align}
for $\forall w>0$. The only additional term compared to proof of Theorem \ref{thm:rgo_chi} is the right-most term above. 
By choosing $w = (1-\alpha)^{-\frac{1-\alpha}{\alpha}} (\eta d)^{1-\alpha}  $, and denote $b = (1-\alpha)^{-\frac{1-\alpha}{\alpha}} $, the above equation becomes
\begin{align}
 \left(1- \frac{ \pi^2 L_\alpha^2 \eta^{1+\alpha} \alpha (1-\alpha)^{ \frac{1-\alpha}{\alpha} } }{ 4 d^{1-\alpha}}\right)^{-\frac{d}{2}} 
 \exp \left( \frac{ \pi^2}{8} L_\alpha^2 \eta^{1 + \alpha} d^\alpha \right)
 \exp\left(\frac{ \|
\iota\|^2_2  }{ \frac{8 b^2 d^{1-\alpha} }{ \pi^2 \alpha \eta^\alpha L_\alpha^2 } - 2 \eta}\right)
\end{align}
Recall that \eqref{eq:real_s_tv} still holds, i.e. $\|\iota \|^2 \le s:= \frac{d^{\frac{1}{2(1+\alpha)}} }{\RtConsteta L_\alpha^{ \frac{1}{1+\alpha}} } 
  $,  since we use the same Algorithm \ref{alg:rgo_appro}.
We further choose $\eta \le  \frac{1 }{2 L_\alpha^\frac{2}{1+\alpha} d^\frac{\alpha}{1+\alpha} } \le \left( \frac{8 \log(1.2) }{\pi^2 L_\alpha^2 d^\alpha } \right)^\frac{1}{1+\alpha} $, 
  so the above equation is bounded by
\begin{align}
2 \exp \left( \frac{\alpha \pi^2 \|\iota\|^2 \eta^\alpha L_\alpha^2 }{ 4 b^2 d^{1-\alpha}} \right) \le 3.
\end{align}
So we obtain $\mE [V] \le 3.$ Next, we bound the term $\mE[(V-\overline V)^2 V^{-1}] $:
\begin{align}
\mE \frac{(V-\overline V)^2}{V} 
\le \mE \frac{(V-\overline V)^2}{V} \one_{V \le \frac{1}{e} } +  \mE \frac{(V-\overline V)^2}{V} \one_{V > \frac{1}{e}}
\le \mE \frac{(V-\overline V)^2}{V} \one_{V \le \frac{1}{e} } +  e \mE (V-\overline V)^2 .
\end{align}
Following the same logic of \eqref{eq:v_diff_squ_appro} in the proof of Theorem \ref{thm:rgo_w2_appro}, we can get that if $$\eta \le \frac{
1
}{
  98 L_\alpha^{\frac{2}{\alpha+1}}d^{\frac{\alpha}{\alpha+1}}
(2+ \log( 1+ \frac{522 }{ \zeta } )) }, $$
then $ e\mE (V - \overline V)^2 \le \zeta/(24)  . $

We then notice that 
since $V \le 1/e$, by the definition of $V$ and Jensen inequality, it holds that $g(x) - \mE[g] \ge 1. $ 
The term $\mE \frac{(V-\overline V)^2}{V} \one_{V \le \frac{1}{e} }$ satisfies
$
\mE \frac{(V-\overline V)^2}{V} \one_{V \le \frac{1}{e} }
\le \frac{1}{e} \mE \left(\frac{V-\overline V}{V}\right)^2 \one_{g(x) - \mE[g] \ge 1}.
$
Moreover, since $g(x) -\mE[g]$ 
is always no less than 1, we can apply the inequality \eqref{eq:const_concentrate_appro}. With $c= 0.2 $ and $\|x_y - w \| \le s$, we have 
\begin{align}\label{eq:cond_conce_appro}
\Pr (\Delta
> r |x) 
\le 1.5
\exp \left( \frac{s^2 }{ 2 d \eta} -c \cdot \frac{g(x) -\mE[g] + r } {L_\alpha^{ \frac{2}{\alpha+1} } d^{\frac{\alpha}{\alpha+1}} \eta}  \right)  
\end{align}
for $\forall r>0$.

Denote $\bar \Delta = \Delta / \log 2 =  (g(z) - g(x) )/\log 2$, and $S = L_\alpha^{ \frac{2}{\alpha+1} } d^{\frac{\alpha}{\alpha+1}} \eta $. We can write 
\begin{align}
 \frac{V-\overline V}{V} 
\overset{\eqref{eq:v_frac}}{ \le } &  \frac{2 \sum_{i=1}^\infty \exp(i \log 2) \Pr ( \bar \Delta \ge i | x) }{ \exp( \frac{-1 + \mE[g] - g(x)}{\log 2} )  [1- \Pr (  \Delta > \log 2 | x) - \Pr (  g(z) - \mE[g] <-1 ) ] } \\
\overset{ \eqref{eq:cond_conce_appro} }{\le} & \frac{ 3 \sum_{i=1}^\infty \exp(i \log 2) \exp( \frac{s^2 }{ 2 d \eta} - c \frac{g(x) - \mE[g] + i\log 2
}{S} ) \exp( \frac{g(x) - \mE[g] }{\log 2} ) }{ \exp(- \frac{1}{\log 2}) (1- 1.5 [ \exp( \frac{s^2 }{ 2 d \eta} - c \frac{g(x) - \mE[g] + \log 2
}{S}  ) + \exp(\frac{s^2 }{ 2 d \eta} - \frac{c}{S} )]) } \\
\le & \frac{ 15 \sum_{i=1}^\infty \exp(i (\log 2- \frac{c \log2
}{S}) ) \exp( \frac{s^2 }{ 2 d \eta} - (\frac{c}{S} - \frac{1}{\log 2}) (g(x) - \mE[g] )  )  }{ 1- 1.5 [ \exp(\frac{s^2 }{ 2 d \eta} - c \frac{g(x) - \mE[g] + \log 2
}{S}  ) + \exp( \frac{s^2 }{ 2 d \eta} - \frac{c}{S} )] } \\
= & \frac{ 15  \exp( \frac{s^2 }{ 2 d \eta} - (\frac{c}{S} - \frac{1}{\log 2}) (g(x) - \mE[g] )  ) / ({\exp( \frac{c \log2
}{S} - \log 2 ) - 1}) }{ 1- 1.5 [ \exp( \frac{s^2 }{ 2 d \eta} - c \frac{g(x) - \mE[g] + \log 2
}{S}  ) + \exp( \frac{s^2 }{ 2 d \eta} - \frac{c}{S} )] } \\
\le & \frac{ 15  \exp(\frac{s^2 }{ 2 d \eta} - (\frac{c}{S} - \frac{1}{\log 2}) 
) / ({\exp( \frac{c \log2
}{S} - \log 2 ) - 1}) }{ 1- 1.5 [ \exp(\frac{s^2 }{ 2 d \eta} - c \frac{ 1
+ \log 2
}{S}  ) + \exp(\frac{s^2 }{ 2 d \eta} - \frac{c}{S} )] } ,
\end{align}
where we use the condition $g(x) - \mE[g] \ge 1$ in the last inequality.
By by the choice $s$ in \eqref{eq:real_s_tv} and 
\begin{align}
\eta \le  ({16 L_\alpha^{ \frac{2}{\alpha+1} } d^{\frac{\alpha}{\alpha+1}} \max \{1.5 + \log( 1 + 7 /\sqrt{\zeta}), \log(30 ) \} ) } )^{-1} ,   
\end{align}
under the condition $g(x) - \mE[g] \ge 1$,
 we guarantee that $1- 1.5 [ \exp(- c \frac{ 1
+ \log 2
}{S}  ) + \exp(- \frac{c}{S} ) \ge 0.8 $ and $ 15  \exp( - (\frac{c}{S} - \frac{1}{\log 2}) 
) / ({\exp( \frac{c \log2
}{S} - \log 2 ) - 1}) \lesssim \zeta $.
So we obtain that $\mE \frac{(V-\overline V)^2}{V} \one_{V \le \frac{1}{e} }
 \le \frac{\zeta}{24} $, and thus $\mE \frac{(V-\overline V)^2}{V}  \le \frac{\zeta}{12}. $

 To bound $\mE \overline V$, we notice that if $\chi_{\pi^{X|Y}}^2(\hat \pi^{X|Y} ) \le 0.5, $ then $\|\hat \pi^{X|Y} - \pi^{X|Y} \|_\TV \le 1, $ thus by \eqref{eq:low_bound_bar_V}, we get $\mE \overline V \ge \frac{1}{2}.$ Combining it with the bounds  $\mE[V] \le 3$ and $\mE \frac{(V-\overline V)^2}{V}  \le \frac{\zeta}{12}, $ we get $\chi_{\pi^{X|Y}}^2( \hat \pi^{X|Y}  )  \le \zeta$.

 The expected number of the iterations for the rejection sampling step is $\cO(1)$ because
the acceptance probability is $\frac{1}{2} \mE[\overline V] $ (Lemma \ref{lem:dist_relation}).
 \end{proof}

 \subsection{Complexity bounds of proximal sampling for semi-smooth potentials}\label{sec:convegence_chi}

 The following Proposition \ref{prop:chi_strongly_conv_appro}, \ref{prop:chi_appro_all} extend the results in Proposition \ref{prop:strongly_convex}, \ref{prop:lsi}, \ref{prop:pi} to $\chi^2$-divergence notion.
 Moreover, they only assume we have inaccurate optimization step \ref{step:opt_inacc}. Our proofs in this section use the same techniques from  \citet[Theorem 5.1-5.4]{Altschuler2023}.

 \begin{proposition}[Convergence in $\chi^2$ for well-conditioned targets]\label{prop:chi_strongly_conv_appro}
Suppose $f$ is $\beta$-strongly convex and $L_1$-smooth. Let $ \delta \in (0, 1 ), \eta =
\tilde \cO \left( {1}/{ (L_1 \sqrt{d}) }  \right) 
$.
Then Algorithm \ref{algo:AlternatingSampler},  with Algorithm \ref{alg:rgo_appro} as RGO step and initialization $x_0\sim \mu_0$, can find a random point $x_T \sim \hat \mu_T$ such that $\chi^2_{\nu} (\hat \mu_T) \le \delta$
in $$T = \cO  \left( \frac{ L_1 \sqrt{d}}{ \beta }  \log \left( \frac{L_1 \sqrt{d}}{ \beta \delta } \right)  \log \left( \frac{ \sqrt{R_{2, \nu} (\mu_0)}}{ \delta} \right) \right) $$ steps.
Furthermore,  each step accesses only $\cO(1)$ many $f(x)$ queries in expectation.
\end{proposition}
\begin{proof}
Recall that the distributions of the iterations $y_t$ and $x_t$ of the ideal proximal sampler are $\psi_t$ 
and $\mu_t$ respectively (see the discussion at the beginning of \S\ref{sec:convergence}).
Denote the $Y$-marginal of $\pi^{XY}$ as  $\pi^Y$.
By analyzing the simultaneous heat flow,
\citet[\S A.4]{chen2022improved}
shows that the forwards step of the proximal algorithm is a
contraction in \Ren divergence, i.e.
\begin{align}\label{eq:forward}
R_{2, \pi^Y}  (\psi_t)  \le \frac{R_{2,\nu} (\mu_t) }{(1+ \eta\beta )^{1/2}} .
\end{align}
We assume that for $\forall y \in \mR^d$, we have 
\begin{align}
\chi_{\pi^{X|Y} (\cdot | y) }^2( \hat \pi^{X|Y} (\cdot | y)   ) \le \zeta   
\end{align}
using $O(1)$ many $f(x)$ queries in expectation. Then according to the data-processing inequality for \Ren divergence \citep[Lemma 2.3]{Altschuler2023}, strong composition rule for \Ren differential privacy~\citep[Lemma 2.9]{altschuler2022privacy}, we have 
\begin{align}
R_{2, \nu}(\hat \mu_{t+1})  
& \le R_{2, \pi^{XY}}  (\text{law} (\hat X_{t+1}, \hat Y_t)) \\
& \le  R_{2, \pi^Y}(\hat \psi_{t})  + \sup_{y_t \in \mR^d } R_{2, \pi^{X|Y} (\cdot | y_t )  } (\hat \pi^{X|Y} (\cdot | y_t )) \\
& \le R_{2, \pi^Y}(\hat \psi_{t}) + \log(1+ \zeta) \label{eq:backward}
\end{align}
Combining \eqref{eq:forward} and \eqref{eq:backward} gives 
\begin{align}
R_{2, \nu}(\hat \mu_{t+1})  \le \frac{R_{2,\nu} (\hat \mu_t) }{(1+ \eta\beta )^{1/2}} + \log(1+ \zeta) . 
\end{align}
Iterating this bound for $T$ times gives 
\begin{align}
R_{2,\nu} (\hat \mu_T ) 
\le \frac{R_{2,\nu} ( \mu_0 ) }{(1+ \eta\beta )^{T/2}} + \log(1+ \zeta) \sum_{t=0}^{T-1} \frac{1}{(1+ \eta\beta)^{t/2}} 
\le \frac{R_{2,\nu} ( \mu_0 ) }{(1+ \eta\beta )^{T/2}} + \frac{\log(1+ \zeta) }{1 - \frac{1}{\sqrt{ 1+ \eta \beta}} }
\end{align}
This error is at most $\delta$ if
\begin{align}
    T = \cO \left( 
    \frac{1}{\log(1+ \eta \beta )} 
    \log \left( \frac{R_{2,\nu}(\mu_0) }{\delta } \right) \right), ~~~ \zeta = 
    \Theta ( \delta 
 \eta \beta).
\end{align}
According to Theorem \ref{thm:rgo_chi_opt}, plugging in the necessary step size 
$\eta = \cO( \frac{1}{ L_1 \sqrt{d} \log( L_1 \sqrt{d} / (\delta \beta)) }) $
finishes the proof.
\end{proof}

\begin{proposition}[Convergence in $\chi^2$ for non-log-concave targets] \label{prop:chi_appro_all}
We use Algorithm \ref{algo:AlternatingSampler} with Algorithm \ref{alg:rgo_appro} as RGO step.
Let $ \delta \in (0,1),$ and $\nu \propto \exp(-f)$.
\\
1)
If $\nu$ satisfies $C_\LSI$-\ref{eq:lsi} and $L_1$-smooth,
then we need \\
$
T = \cO  \left( \frac{ L_1 \sqrt{d}}{ C_\LSI }  \log \left( \frac{L_1 \sqrt{d}}{ C_\LSI \delta } \right)  \log \left( \frac{ { R_{2, \nu} (\mu_0)}}{ \delta} \right) \right)  
$  to have $\chi^2_{\nu} (\hat \mu_T) \le \delta $.  \\
2) If $\nu$ satisfies $C_\PI$-\ref{eq:pi} and $f$ is $L_\alpha$-$\alpha$-semi-smooth, then 
we need \\
$
T =   \cO \left( 
 \frac{L_\alpha^{\frac{2}{\alpha+1}}d^{\frac{\alpha}{\alpha+1}} }{C_\PI}\log \left(\frac{L_\alpha^{\frac{2}{\alpha+1}} d^{\frac{\alpha}{\alpha+1}} 
}{ C_\PI \delta } \right) \log \left( \frac{\chi_\nu^2(\mu_0)}{\delta } \right) \right)
$ to have $\chi^2_{\nu} (\hat \mu_T) \le \delta $. \\
Furthermore,  each step accesses only $\cO(1)$ many $f(x)$ queries in expectation.
\end{proposition}
\begin{proof}
 1) The proof is the same as of Proposition \ref{prop:chi_strongly_conv_appro} since \eqref{eq:forward} also holds under LSI: 
 $$R_{2, \pi^Y}  (\psi_t)  \le \frac{R_{2,\nu} (\mu_t) }{(1+ \eta C_\LSI )^{1/2}} .$$

 2) \citet[\S A.4]{chen2022improved} shows 
\begin{align}\label{eq:forward_pi}
\chi^2_{ \pi^Y}  (\psi_t)  \le \frac{\chi^2_{\nu} (\mu_t) }{(1+ \eta C_\PI)} .
\end{align}
In the same time, \eqref{eq:backward} still holds as it does not use the LSI assumption. 
Recall that $R_{2,\nu} = \log(1+\chi^2_\nu) $.
Combining \eqref{eq:forward_pi} and \eqref{eq:backward} gives
\begin{align}
\chi^2_{\nu}(\hat \mu_{t+1})  \le 
\left( \frac{1+\zeta}{1 + \eta C_\PI} \right) \chi^2_{\nu} (\hat \mu_t)  + \zeta .
\end{align}
Iterating this bound for $T$ times gives 
\begin{align}
\chi^2_{\nu} (\hat \mu_T ) 
\le \left(\frac{1+\zeta }{1+ \eta\beta }\right)^{T} \chi^2_{\nu} ( \mu_0 ) + \zeta \sum_{t=0}^{T-1} \left( \frac{1+\zeta}{1+\eta C_\PI} \right)^t 
\le \left(\frac{1+\zeta }{1+ \eta\beta }\right)^{T} \chi^2_{\nu} ( \mu_0 ) + \frac{\zeta }{1 - \frac{1 + \zeta }{{ 1+ \eta C_\PI }} }
\end{align}
This error is at most $\delta$ if we choose 
\begin{align}
    T = \cO \left( 
    \frac{1}{\log((1+ \eta C_\PI )/(1+\zeta) )} 
    \log \left( \frac{\chi^2_{\nu}(\mu_0) }{\delta } \right) \right), ~~~ \zeta \le 
    \frac{\delta \eta C_\PI }{2(1+ \eta C_\PI + \delta / 2)}.
\end{align}
Plugging in the choice of $\eta$ in Theorem \ref{thm:rgo_chi_opt} finishes the proof.

\end{proof}

\begin{remark}[Convergence stability]
In the convergence analysis, an interesting question is whether we can allow the number of iterations to go to infinity. Lemma \ref{lem:tv_telecope}, \ref{lem:w2_telecope} do not guarantee stability since the RGO accumulation error linearly depends on the iterations $T$. However, the proof of Lemma \ref{lem:w2_telecope} can be easily strengthened when $\pi$ is strongly-log-concave to bound accumulation error. Moreover, with 
the techniques in Theorem 5.1-5.4 of \citet{Altschuler2023}, our Proposition \ref{prop:chi_strongly_conv_appro}, \ref{prop:chi_appro_all} now ensure stable convergence.
\end{remark}

 \end{appendix}
\end{document}